\documentclass[reqno, 12pt]{amsart}
\usepackage{amsmath,amsxtra,amssymb,latexsym, amscd,amsthm}
\usepackage{graphicx}
\usepackage[utf8]{inputenc}
\usepackage[mathscr]{euscript}
\usepackage{mathrsfs}
\usepackage[english]{babel}
\usepackage{color}
\usepackage{tikz}
\usepackage{hyperref}
\usepackage{enumerate}
\usepackage{graphicx} 
\usepackage[T1]{fontenc}

\newtheorem {theorem}{Theorem}[section]
\newtheorem {corollary}{Corollary}[section]
\newtheorem {proposition}{Proposition}[section]
\newtheorem {lemma}{Lemma}[section]
\newtheorem {example}{Example}[section]
\newtheorem {definition}{Definition}[section]
\newtheorem {remark}{Remark}[section]

\def\Limsup{\mathop{{\rm Lim}\,{\rm sup}}}

\def\ow{o\kern-.42em\raise.82ex\hbox{
   \vrule width .12em height .0ex depth .075ex \kern-0.16em \char'56}\kern-.07em}
\def\OW{o\kern-.460em\raise1.36ex\hbox{
\vrule width .13em height .0ex depth .075ex \kern-0.16em
\char'56}\kern-.07em}
\def\DD{D\kern-.7em\raise0.4ex\hbox{\char '55}\kern.33em}

\title{Exact penalty functions in optimization with unbounded constraint sets}

\author[Liguo Jiao]{Liguo Jiao}
\address{Academy for Advanced Interdisciplinary Studies, Northeast Normal University, Changchun, 130024, Jilin Province, China}
\email{jiaolg356@nenu.edu.cn; hanchezi@163.com}

\author[Ti\'{\^{e}}n-S\OW n Ph\d{a}m]{Ti\'{\^{e}}n-S\OW n Ph\d{a}m$^\dagger$}
\address{Department of Mathematics, Dalat University, 1 Phu Dong Thien Vuong, Dalat, Vietnam}
\email{sonpt@dlu.edu.vn}

\author[Nguyen Van Tuyen]{Nguyen Van Tuyen}
\address{Department of Mathematics, Hanoi Pedagogical University 2, Xuan Hoa, Phuc Yen, Vinh Phuc, Vietnam}
\email{nguyenvantuyen83@hpu2.edu.vn}

\date{\today}
\subjclass[2010]{Primary 90C26; Secondary 49J52, 14P10}

\keywords{constrained optimization, exact penalty, calmness, semi-algebraic, \L ojasiewicz inequality}

\thanks{The work was partially supported by the Chinese National Natural Science Foundation under grant numbers 12471478, 12371300.}
\thanks{$^\dagger$Corresponding author}

\begin{document}

\begin{abstract}
This paper identifies necessary and sufficient conditions for the exactness of penalty functions in optimization problems whose constraint sets are not necessarily bounded.  The case where the data of problems are {\em locally Lipschitz or semi-algebraic} is studied in detail.
The conditions are given in terms of properties of the objective and residual functions of the problems in question. 
The obtained results generalize and improve some known results in the literature on exact penalty functions. 
\end{abstract}

\maketitle

\section{Introduction}		

Constrained optimization problems arise in various fields, including engineering, economics, and machine learning, where the goal is to optimize an objective function subject to a set of constraints. Problems with complicated constraints are very difficult to deal with.
A common approach for solving such problems is the usage of penalty function methods, wherein a constrained optimization problem is transformed into a sequence of unconstrained optimization problems whose solutions ideally converge to a solution of the original constrained problem; the unconstrained problems are formed by bringing the constraints into the objective function via a residual function and a scalar penalty parameter. 
Among these methods, {\em exact penalty functions} are particularly important because they allow us to solve a constrained optimization problem by solving a {\em single unconstrained} optimization problem.  
More precisely, a penalty function is said to be {\em exact} (or to have the {\em exact penalty property}) if there is a penalty parameter for which the corresponding unconstrained optimization problem and the original constrained optimization problem have the same optimal value and the same optimal solutions.

The concept of exact penalization traces back to the seminal works of Eremin~\cite{Eremin1966} and Zangwill~\cite{Zangwill1967} (see also \cite{Pietrzykowski1969}). There is a huge literature on all aspects of the theory and applications of exact penalty functions; for more details, we refer the reader to the monographs \cite{CuiPang2021, Fiacco1990, Hiriart-Urruty1993book, Ioffe2017, Luo1996, Mordukhovich2018, Rockafellar1998, Zaslavski2010}, the surveys \cite{Boukari2010, Burke1991-2, Fletcher1983} and the recent papers \cite{LGuo2018, Jiang2023, Qian2024, Xiao2021} with the references therein.

In this paper we would like to identify necessary and sufficient conditions for a penalty function to be exact; the conditions are given in terms of properties of the objective and residual functions of the problems in question. As far as we know, for optimization problems whose constraint sets are {\em unbounded}, there are not many such studies; see \cite{Bertsekas1975, Contaldi1993, Dolgopolik2016, Zaslavski2005, Zaslavski2010, Zaslavski2013}.

\subsection*{Contributions}
We mainly study the exactness of penalty functions for constrained optimization problems with potentially unbounded constraint sets. Among other things, our main contributions are given in several steps:

\begin{itemize}
\item A necessary and sufficient condition for a penalty function to be exact is given (see Theorem~\ref{DL31}) and an equivalent relation between the exactness of a penalty function and the global calmness for the problem in question is established (see Theorem~\ref{DL32}). The latter can be seen as a global version of \cite[Proposition~6.4.3]{Clarke1990} and \cite[Theorem 1.1]{Burke1991-1};

\item For problems defined by locally Lipschitz functions, under some regularity conditions, a new necessary and sufficient condition for a penalty function to be exact is derived (see Theorem~\ref{DL41}); this result improves \cite[Theorems~1.5~and~3.1]{Zaslavski2005} as well as \cite[Theorems~1.3]{Zaslavski2013} (in the finite-dimensional setting);

\item Based on the study of {\L}ojasiewicz inequalities on unbounded sets, exact penalty functions for constrained optimization problems with semi-algebraic data are proposed. In the semi-algebraic setting, the obtained results generalize \cite[Theorem~1]{Warga1992}, \cite[Theorem~3.1]{Dedieu1992} and \cite[Theorem~2.1.2]{Luo1996}, wherein optimization problems over compact sets with subanalytic data are studied. 
It should be noted that the cited theorems could fail to hold when the compact assumption is absent (see Examples~\ref{VD53}~and~\ref{VD55}). 
\end{itemize}

At this point, we would like to note the following facts:
\begin{itemize}
\item the closedness of constraint sets and the continuity of objective functions, except for Theorem~\ref{DL41}, are not required;

\item the boundedness of constraint sets is not imposed; 

\item the existence of optimal solutions is not assumed.
\end{itemize}

We confine our study to the finite-dimensional case for two reasons. First of all, to lighten the exposition, we would like to concentrate on the basic ideas without technical and notational complications. Furthermore, the results obtained in Section~\ref{Section5} are peculiar to finite-dimensions.

The tools used in this study come from variational analysis and semi-algebraic geometry.

The rest of the paper is organized as follows.
Some definitions and preliminary results from variational analysis and semi-algebraic geometry are recalled in Section~\ref{Section2}. The results and their proofs are presented in Sections~\ref{Section3}, \ref{Section4}, and \ref{Section5}. Conclusions are given in Section~\ref{Section6}.

\section{Preliminaries} \label{Section2}

\subsection{Notation} 
Throughout this work we deal with the Euclidean space $\mathbb{R}^n$ equipped with the usual scalar product $\langle \cdot, \cdot \rangle$ and the corresponding norm $\| \cdot\|.$ We denote by $\mathbb{B}_r(x)$ the closed ball centered at $x$ with radius $r;$  when ${x}$ is the origin of $\mathbb{R}^n$ we write $\mathbb{B}_{r}$ instead of $\mathbb{B}_{r}({x}),$ and when $r = 1$ we write  $\mathbb{B}$ instead of $\mathbb{B}_{1}.$ We will adopt the convention that $\inf \emptyset = \infty$ and $\sup \emptyset = -\infty;$ the notation $x  \to \infty $ means that $\|x\| \to \infty$.

Let $\mathbb{R}_+ := [0, \infty), \mathbb{R}_- := (-\infty, 0]$ and $\overline{\mathbb{R}} := \mathbb{R} \cup \{\infty\}.$ 
For a real number $r,$ we write $[r]_+ := \max\{r, 0\}.$ 

For a nonempty set $\Omega \subset \mathbb{R}^n,$ the closure, convex hull and conic hull of $\Omega$ are denoted, respectively, by 
$\mathrm{cl}\, {\Omega},$ $ \mathrm{co}\, \Omega$ and $\mathrm{cone}\, \Omega.$ 
We will associate with $\Omega$ the {\em distance function}
$$\mathrm{dist}(\cdot, \Omega) \colon \mathbb{R}^n \to \mathbb{R}, \quad x \mapsto \inf_{x' \in \Omega} \|x - x'\|,$$
and define the {\em Euclidean projector of} $x \in \mathbb{R}^n$ to $\Omega$ by
\begin{eqnarray*}
\Pi_{\Omega}(x) &:=& \{x' \in \Omega \ | \  \|x - x' \| = \mathrm{dist}(x, \Omega)\}.
\end{eqnarray*}
The {\em indicator function} $\delta_{\Omega} \colon \mathbb{R}^n \to \overline{\mathbb{R}}$ of $\Omega$ is defined by
\begin{eqnarray*}
\delta_\Omega(x) &:=&
\begin{cases}
0 & \textrm{ if } x \in \Omega, \\
\infty & \textrm{ otherwise.}
\end{cases}
\end{eqnarray*}

For an extended real-valued function $f \colon \mathbb{R}^n \rightarrow \overline{\mathbb{R}},$ we denote its {\em effective domain}, {\em graph,}  and {\em epigraph} by,  respectively,
\begin{eqnarray*}
\mathrm{dom} f &:=& \{ x \in \mathbb{R}^n \ | \ f(x) < \infty  \},\\
\mathrm{gph} f &:=& \{ (x, y) \in \mathbb{R}^n \times \mathbb{R} \ | \ f(x)  = y \},\\
\mathrm{epi} f &:=& \{ (x, y) \in \mathbb{R}^n \times \mathbb{R} \ | \ f(x) \le y \}.
\end{eqnarray*}
We call $f$ a {\em proper} function if $f(x) < \infty$ for at least one $x \in \mathbb{R}^n,$ or in other words, if $\mathrm{dom} f$ is a nonempty set.
The function $f$ is said to be {\em lower semicontinuous} if its epigraph is a closed set. 

For an extended real-valued function $f \colon \mathbb{R}^n \rightarrow \overline{\mathbb{R}}$ and a set $\Omega \subset \mathbb{R}^n,$  we let
\begin{eqnarray*}
\inf_\Omega f &:=& \inf_{x \in \Omega}  f(x) \ := \  \inf \{ f(x) \mid x \in \Omega\}.
\end{eqnarray*}
We introduce notation also for the set of points $x$ where the minimum of $f$ over $\Omega$ is regarded as
being attained:
\begin{eqnarray*}
\mathrm{argmin}_\Omega f(x) &:= &
\begin{cases}
\{x \in \Omega \mid f(x) = \inf_{x \in \Omega} f(x)\} & \textrm{ if } \inf_{x \in \Omega} f(x) \ne \infty, \\
\emptyset & \textrm{ otherwise.}
\end{cases}
\end{eqnarray*}

For a set-valued mapping $F \colon \mathbb{R}^n \rightrightarrows \mathbb{R}^m,$ we denote its {\em graph} by
\begin{eqnarray*}
\mathrm{gph} F&:=& \{(x,y)\in \mathbb{R}^n \times \mathbb{R}^m \mid y \in F(x)\};
\end{eqnarray*}
the {\em Painlev\'e--Kuratowski outer limit} of $F$ at $x \in \mathbb{R}^n$ is defined by
\begin{eqnarray*}
\Limsup_{x' \to {x}} F(x) &:=& \{y \in \mathbb{R}^m \mid \exists x_k \to {x}, \exists y_k \in F(x_k), y_k \to y\}.
\end{eqnarray*}

\subsection{Normal cones, subdifferentials and coderivatives}

Here we recall some definitions and results of variational analysis, which can be found in~\cite{Mordukhovich2006, Mordukhovich2018, Rockafellar1998}.

\begin{definition}{\rm Consider a set $\Omega\subset\mathbb{R}^n$ and a point ${x} \in \Omega.$
The {\em limiting normal cone} (also known as the {\em basic} or {\em Mordukhovich normal cone}) to $\Omega$ at ${x}$
is defined by
\begin{eqnarray*}
N_{\Omega}(x) &:=& \Limsup_{x' \rightarrow x} \Big[\mathrm{cone} \big(x' - \Pi_{\Omega}(x') \big) \Big].
\end{eqnarray*}
If $x \not \in \Omega,$ we put ${N}_\Omega({x}) := \emptyset.$
}\end{definition}

\begin{remark}{\rm
(i) It is well known that $N_\Omega({x})$ is a closed (possibly non-convex) cone.

(ii) If $\Omega$ is a manifold of class $C^1,$ then the normal cone $N_\Omega({x})$ is equal to the normal space to $\Omega$ at ${x}$ in the sense of differential geometry; see \cite[Example~6.8]{Rockafellar1998}. 
}\end{remark}

\begin{definition}{\rm
Consider an extended real-valued function $f\colon\mathbb{R}^n \to \overline{\mathbb{R}}$ and a point ${x} \in \mathrm{dom} f.$ The {\em limiting} (or {\em Mordukhovich}) {\em subdifferential} of $f$ at ${x}$ is defined by
\begin{eqnarray*}
\partial f({x}) &:=& \{ u \in \mathbb{R}^n \mid (u,-1)\in {N}_{\mathrm{epi} f}({x},f({x}))  \}.
\end{eqnarray*}
If $x \not \in \mathrm{dom} f,$ we put $\partial f({x}) := \emptyset.$
}\end{definition}

\begin{remark}{\rm
In \cite{Mordukhovich2006, Mordukhovich2018, Rockafellar1998} the reader can find equivalent analytic descriptions of the limiting subdifferential $\partial f({x})$ and comprehensive studies of it and related constructions. For convex $f,$ this subdifferential coincides with the convex subdifferential. Furthermore, if the function $f$ is of class $C^1,$ then $\partial f({x}) = \{\nabla f({x})\}.$
}\end{remark}

\begin{definition}{\rm 
Consider a set-valued mapping $F \colon \mathbb{R}^n \rightrightarrows \mathbb{R}^m$ and a point $({x}, {y}) \in \mathrm{gph} F.$ {\em The (basic) coderivative} of $F$ at $({x}, {y})$ is the set-valued mapping $D^*F({x}, {y}) \colon \mathbb{R}^m \rightrightarrows \mathbb{R}^n$ defined by
\begin{eqnarray*}
D^*F({x}, {y})(v) &:=& \{u \in \mathbb{R}^n \mid (u, -v)\in N_{\mathrm{gph} F} ({x}, {y}) \} \quad \textrm{ for all } \quad v \in \mathbb{R}^m.
\end{eqnarray*}
}\end{definition}

The following facts are well known (see, for example, \cite{Mordukhovich2018}).

\begin{lemma}\label{BD21}
Consider a nonempty closed set $\Omega\subset\mathbb{R}^n.$ We have for all ${x} \in \Omega,$
\begin{eqnarray*}
N_{\Omega}({x}) &=& \Limsup_{x' \xrightarrow{\Omega} {x}} {N}_{\Omega}(x'),
\end{eqnarray*}
where $x' \xrightarrow{\Omega} {x}$ means that $x' \rightarrow {x} $ with $x' \in \Omega.$ 
\end{lemma}

\begin{lemma} \label{BD22}
Let $\Omega \subset \mathbb{R}^n$ be a nonempty closed set. Then
\begin{eqnarray*}
\partial \mathrm{dist}(\cdot, \Omega)(x) &=& 
\begin{cases}
N_{\Omega}(x) \cap \mathbb{B} & \quad \textrm{ if } x \in \Omega, \\
\frac{x - \Pi_{\Omega}(x)}{\mathrm{dist}(x, \Omega)} & \quad \textrm{ otherwise.}
\end{cases}
\end{eqnarray*}
\end{lemma}

\begin{lemma} \label{BD23}
For a lower semicontinuous function $f \colon \mathbb{R}^n \to \overline{\mathbb{R}}$ and a point $\overline{x} \in \mathrm{dom} f,$ we have
\begin{eqnarray*}
\partial f(\overline{x}) &=& \Limsup_{x \xrightarrow{f} \overline{x}} {\partial} f(x),
\end{eqnarray*}
where $x \xrightarrow{f} \overline{x}$ means that $x  \to \overline{x}$ and $f(x) \to f(\overline{x}).$
\end{lemma}

\begin{lemma}\label{SumRule}
Let $f_i \colon \mathbb{R}^n \to \overline{\mathbb{R}}$, $i=1,\dots,m$ with $m \geq 2$, be lower semicontinuous at $\overline{x} \in \mathbb{R}^n$  and let all but one of these functions be Lipschitz continuous around $\overline{x}.$ Then the following inclusions hold
\begin{eqnarray*}
\partial \left( f_1+\cdots+f_m\right)(\overline{x}) & \subset & \partial f_1(\overline{x})+\cdots+\partial f_m(\overline{x}), \\
\partial (\max f_i) (\overline{x}) & \subset & \bigcup \mathrm{co}(\{v_i \mid i \in I(\overline{x})\}), 
\end{eqnarray*}
where $I(\overline{x}) := \{ i \in \{1,\ldots,m\}\mid f_i(\overline{x})= \max_j f_j (\overline{x})\}$ and the union is taken over all vectors $v_i \in \partial f_i(\overline{x}) $ for $i \in I(\overline{x}).$
\end{lemma}

\begin{lemma}[Fermat rule] \label{FermatRule}
If a proper function $f \colon \mathbb{R}^n \to \overline{\mathbb{R}}$ has a local minimum at $\overline{x} \in \mathrm{dom} f,$  then $0 \in \partial f(\overline{x}).$
\end{lemma}

\begin{lemma} \label{BD26}
Let $f \colon \mathbb{R}^n \to \overline{\mathbb{R}}$ be Lipschitz continuous around $\overline{x} \in \mathbb{R}^n$ with modulus $L.$ Then $\partial f(\overline{x}) \ne \emptyset$ and
\begin{eqnarray*}
\|v\| &\le& L \quad \textrm{ for all } \quad v \in \partial f(\overline{x}).
\end{eqnarray*}
\end{lemma}

\begin{lemma}[chain rule] \label{ChainRule}
Let $g \colon \mathbb{R}^n \to {\mathbb{R}}^m$ be Lipschitz continuous around $\overline{x} \in \mathbb{R}^n$ and $h \colon \mathbb{R}^m \to {\mathbb{R}}$ be Lipschitz continuous around $\overline{y} := g(\overline{x}) \in \mathbb{R}^m.$ Then
\begin{eqnarray*}
\partial (h \circ g) (\overline{x}) &\subset& \bigcup_{w \in \partial h(\overline{y})} \partial \langle w, g\rangle (\overline{x}).
\end{eqnarray*}
\end{lemma}

Finally, the next lemma expresses the coderivative of a single-valued Lipschitz mapping $g \colon \mathbb{R}^n \to {\mathbb{R}}^m$ via the subdifferential of the scalarization $\langle w, g\rangle \colon \mathbb{R}^n \to {\mathbb{R}}, x \mapsto \langle w, g(x)\rangle.$ 

\begin{lemma} \label{BD28}
Let $g \colon \mathbb{R}^n \to {\mathbb{R}}^m$ be Lipschitz continuous around $\overline{x} \in \mathbb{R}^n.$ Then
\begin{eqnarray*}
D^* g(\overline{x})(w) &=& \partial \langle w, g\rangle (\overline{x}) \quad \textrm{ for all } \quad w \in \mathbb{R}^m.
\end{eqnarray*}
\end{lemma}

\subsection{Semi-algebraic geometry}
Now, we recall some notions and results of semi-algebraic geometry, which can be found in \cite{Benedetti1990, Bochnak1998, HaHV2017}.

\begin{definition}{\rm
A subset $S$ of $\mathbb{R}^n$ is {\em semi-algebraic} if it is a finite union of sets of the form
$$\{x \in \mathbb{R}^n \ | \  f_i(x) = 0, \ i = 1, \ldots, k; f_j(x) > 0, \ j = k + 1, \ldots, p\},$$
where all $f_{i}$ are polynomials.
In other words, $S$ is a union of finitely many sets, each defined by finitely many polynomial equalities and inequalities.

A function $f \colon S \rightarrow {\mathbb{R} \cup \{\infty\}}$ is said to be {\em semi-algebraic} if its graph
\begin{eqnarray*}
\{ (x, y) \in S \times \mathbb{R} \ | \ y =  f(x) \}
\end{eqnarray*}
is a semi-algebraic set.
}\end{definition}

\begin{example}{\rm
Each polyhedral set is the intersection of a finite number of closed half-spaces or hyperplanes, and so is semi-algebraic.
}\end{example}

A major fact concerning the class of semi-algebraic sets is its stability under linear projections.

\begin{theorem}[Tarski--Seidenberg theorem] \label{TarskiSeidenbergTheorem}
The image of any semi-algebraic set $S \subset \mathbb{R}^n$ under a projection to any linear subspace of $\mathbb{R}^n$ is a semi-algebraic set.
\end{theorem}

\begin{remark}{\rm
As an immediate consequence of the Tarski--Seidenberg theorem, we get the semi-algebraicity of any set $\{ x \in A \ | \  \exists y \in B,  (x, y) \in C \},$  provided that $A ,  B,$  and $C$  are semi-algebraic sets in the corresponding spaces. Also, $\{ x \in A \ | \  \forall y \in B,  (x, y) \in C \}$ is a semi-algebraic set as its complement is the union of the complement of $A$  and the set $\{ x \in A \ | \  \exists y \in B,  (x, y) \not\in C \}.$
Thus, if we have a finite collection of semi-algebraic sets, then any set obtained from them with the help of a finite chain of quantifiers on variables is also semi-algebraic.
}\end{remark}

The following well known lemmas will be of great importance for us (see, for example, \cite{Benedetti1990, Bochnak1998, HaHV2017}).

\begin{lemma}\label{FiniteLemma}
Every semi-algebraic set in $\mathbb{R}$ is a finite union of points and intervals.
\end{lemma}

\begin{lemma}[curve selection lemma at infinity]\label{CurveSelectionLemma}
Let $A\subset \mathbb{R}^n$ be a semi-algebraic set, and let $f := (f_1, \ldots,f_p) \colon  \mathbb{R}^n \to \mathbb{R}^p$ be a semi-algebraic map. Assume that there exists a sequence $\{x^\ell\}$ such that $x^\ell \in A$, $\lim_{l \to \infty} \| x^\ell  \| = \infty$ and $\lim_{l \to \infty} f(x^\ell)  = y \in(\overline{\mathbb{R}})^p,$ where $\overline{\mathbb{R}} := \mathbb{R} \cup \{\pm \infty\}.$ Then there exists a smooth semi-algebraic curve $\varphi \colon (0, \epsilon)\to \mathbb{R}^n$ such that $\varphi(t) \in A$ for all $t \in (0, \epsilon), \lim_{t \to 0} \|\varphi(t)\| = \infty,$ and $\lim_{t \to 0} f(\varphi(t)) = y.$
\end{lemma}

\begin{lemma}[growth dichotomy lemma] \label{GrowthDichotomyLemma}
\begin{enumerate}[\upshape (i)]
\item Let $f \colon (0, \epsilon) \rightarrow {\mathbb R}$ be a semi-algebraic function with $f(t) \ne 0$ for all $t \in (0, \epsilon).$ Then there exist constants $a \ne 0$ and $\alpha \in {\mathbb Q}$ such that $f(t) = at^{\alpha} + o(t^{\alpha})$ as $t \to 0^+.$

\item Let $f \colon (r, +\infty) \rightarrow {\mathbb R}$ be a semi-algebraic function with $f(t) \ne 0$ for all $t \in (r, +\infty).$ Then there exist constants $b \ne 0$ and $\beta \in {\mathbb Q}$ such that $f(t) = bt^{\beta} + o(t^{\beta})$ as $t \to +\infty.$
\end{enumerate}
\end{lemma}

\begin{lemma}[monotonicity lemma]\label{MonotonicityLemma}
Let $f\colon (a, b) \to \mathbb{R}$ be a semi-algebraic function.
Then there are $a = a_0 < a_1 < \cdots< a_s < a_{s+1} = b$ such that, for each $i = 0,\ldots, s,$ the restriction $f|_{(a_i,a_{i+1})}$ is analytic and either constant, strictly increasing, or strictly decreasing.
\end{lemma}

\section{Exact penalty functions: the general case} \label{Section3}

Throughout this paper, we consider the following constrained optimization problem
\begin{equation} \label{Problem}
\textrm{minimize } f(x) \quad \textrm{ subject to } x \in S, \tag{P}
\end{equation}
where $S$ is a nonempty subset of $\mathbb{R}^n$ and $f \colon \mathbb{R}^n \rightarrow \overline{\mathbb{R}}$ is an extended real-valued function. To avoid triviality, we assume in what follows that the optimal value $f_* := \inf_{x \in S} f(x)$ of
the problem~\eqref{Problem} is finite. In this section, we will present some global necessary and sufficient condition for a penalty function to be exact.
Let us start with the following.

\begin{definition}{\rm
A {\em residual function} of $S$ is a nonnegative valued function $\psi \colon \mathbb{R}^n \to \mathbb{R}_+$ such that $\psi(x) = 0$ if and only if $x \in S.$ 
}\end{definition}

\begin{remark}{\rm
If $\psi \colon \mathbb{R}^n \to \mathbb{R}_+$ is a residual function for $S,$ then for any function $\phi \colon \mathbb{R}^n \to (0, \infty)$ and any constants $\alpha > 0$ and $\beta > 0,$ one has the functions $\phi \cdot \psi,$ $\psi^\alpha$ and $\psi^\alpha + \psi^\beta$ are residual for $S.$ 
}\end{remark}

The following simple fact will play an important role in establishing exact penalty properties in optimization problems with unbounded constraint sets.

\begin{theorem} \label{DL31}
For any residual function $\psi \colon \mathbb{R}^n \to \mathbb{R}_+$ of $S,$ 
the following three properties are equivalent:
\begin{enumerate}[{\rm (i)}]
\item for all $c$ sufficiently large$,$ one has 
\begin{eqnarray*}
\inf_{x \in S} f(x) &=& \inf_{x \in \mathbb{R}^n} \big\{ f(x) + c\,\psi(x) \big\};
\end{eqnarray*}

\item there exists a constant $c_* > 0$ such that 
\begin{eqnarray*}
\inf_{x \in S} f(x) &=& \inf_{x \in \mathbb{R}^n} \big\{ f(x) + c_*\,\psi(x) \big\};
\end{eqnarray*}

\item there exists a constant $c_* > 0$ such that 
\begin{eqnarray} \label{Eqn1}
c_* \psi(x) &\ge& [f_* - f(x)]_+ \quad \textrm{ for all } \quad x \in \mathbb{R}^n.
\end{eqnarray}
\end{enumerate}
When these equivalent properties hold$,$ one has for all $c > c_*$ that
\begin{eqnarray*}
\mathrm{argmin}_S f(x) &=& \mathrm{argmin}_{\mathbb{R}^n} \big\{ f(x) + c\, \psi (x) \big\}.
\end{eqnarray*}
\end{theorem}

\begin{proof}
The implications (i) $\Rightarrow$ (ii) $\Rightarrow$ (iii) is obvious.

(iii) $\Rightarrow$ (i) Fix $c > c_*$ and take any $x \in \mathbb{R}^n.$ 
If $f(x) \ge f_*,$ then $f(x) + c\, \psi(x)   \ge f_*;$ otherwise \eqref{Eqn1} gives
\begin{eqnarray*}
f(x) + c\, \psi(x) &\ge& f(x) + c_*\, \psi(x) \ \ge \ f(x) + [f_* - f(x)]_+ \ = \ f_*.
\end{eqnarray*}
In both cases we have 
\begin{eqnarray*}
f(x) + c\, \psi(x) &\ge& f_*.
\end{eqnarray*}
Since $x \in \mathbb{R}^n$ was chosen arbitrarily, we conclude that 
\begin{eqnarray*}
\inf_{x \in \mathbb{R}^n} \big\{ f(x) + c\, \psi (x) \big\} &\ge& f_*.
\end{eqnarray*}
On the other hand, by definition, we have
\begin{eqnarray*}
\inf_{x \in \mathbb{R}^n} \big\{ f(x) + c\, \psi (x) \big\} & \le & \inf_{x \in S} \big\{ f(x) + c\, \psi (x) \big\} \ = \ \inf_{x \in S} f(x) \ = \ f_*.
\end{eqnarray*}
Therefore, (i) holds.

To show the last conclusion, assume that $x_0$ minimizes $f$ on $S.$ Then 
\begin{eqnarray*}
f(x_0) + c\, \psi(x_0) &=& f(x_0) \ = \ f_*,
\end{eqnarray*}
which, together with (i), yields that $x_0$ minimizes $f + c\psi$ on $\mathbb{R}^n.$

Conversely, assume that $x_0$ minimizes $f + c\psi$ on $\mathbb{R}^n.$ 
Since $c > c_*$ and the function $\psi$ is nonnegative, it holds that
\begin{eqnarray*}
f (x_0) + c_* \psi(x_0) \leq f(x_0) + c\, \psi(x_0) &\leq& f_*.
\end{eqnarray*}
On the other hand, it follows from \eqref{Eqn1} that
\begin{eqnarray*}
f(x_0) + c_*\psi(x_0)  &\ge& f_*.
\end{eqnarray*}
Therefore,
\begin{eqnarray*}
f(x_0) + c_*\psi(x_0) &=& f(x_0) + c\, \psi(x_0)  \ = \ f_*.
\end{eqnarray*}
Since $c_* < c,$ then $\psi(x_0) = 0.$ Hence $x_0 \in S$ and $f(x_0) = f_*.$ Thus, $x_0$ minimizes $f$ on $S.$
\end{proof}

\begin{example}{\rm
Let $n := 1, S := \{x \in \mathbb{R} \mid x < 1\}$ and $f(x) := |x| + \delta_{\{x \ne 0\}}(x).$
By definition, the function $\psi \colon \mathbb{R} \to \mathbb{R}$ defined by
$$\psi(x) := 
\begin{cases}
\max\{x - 1, 0\} & \textrm{ if } x \ne 1, \\
1 & \textrm{ otherwise}
\end{cases}$$
is a residual function for $S.$ Fix any $c_* > 0.$ It is easy to check that $f_* = 0$ and 
\begin{eqnarray*}
c_* \psi(x) &\ge& 0 \ = \ [f_* - f(x)]_+ \quad \textrm{ for all } \quad x \in \mathbb{R}.
\end{eqnarray*}
By Theorem~\ref{DL31}, for any $c > c_*$ we have
\begin{eqnarray*}
\inf_{x \in S} f(x) &=& \inf_{x \in \mathbb{R}} \big\{ f(x) + c\,\psi(x) \big\},\\
\mathrm{argmin}_S f(x) &=& \mathrm{argmin}_{\mathbb{R}} \big\{ f(x) + c\, \psi (x) \big\} \ = \ \emptyset.
\end{eqnarray*}
Also observe that the constraint set $S$ is neither closed nor bounded and the objective function $f$ is not lower semicontinuous.
}\end{example}

The results presented in the rest of this section are inspired by the works of Clarke \cite[Section~6.4]{Clarke1990} and Burke \cite{Burke1991-1, Burke1991-2}. To begin with, assume that
\begin{eqnarray*}
S &:=& \{x \in \mathbb{R}^n \mid g(x) \in C\},
\end{eqnarray*}
where $g \colon \mathbb{R}^n \rightarrow \mathbb{R}^m$ is a mapping and $C$ is a nonempty subset of $\mathbb{R}^m.$ Then the problem~\eqref{Problem} can be rewritten as
\begin{equation*} \label{Problem0}
\textrm{minimize } f(x) \quad \textrm{ subject to } g(x) \in C. \tag{P'}
\end{equation*}
In this case, an equivalent relation between the exact penalization and the global calmness can be stated as follows; see also  \cite[Proposition~6.4.3]{Clarke1990} and \cite[Theorem~1.1]{Burke1991-1}.

\begin{theorem} \label{DL32}
The following two properties are equivalent:
\begin{enumerate}[{\rm (i)}]
\item for all $c$ sufficiently large one has
\begin{eqnarray*}
\inf_{g(x) \in {C}} f(x) &=& \inf_{x \in \mathbb{R}^n} \big\{ f(x) + c\, \mathrm{dist} \big(g(x), {C}\big) \big\};
\end{eqnarray*}

\item the problem~\eqref{Problem0} is globally calm in the sense that there exists a constant $c_* > 0$ such that for every pair $(\overline{x}, \overline{u}) \in \mathbb{R}^n \times \mathbb{R}^m$ with $g(\overline{x}) \in {C} + \overline{u},$ one has
\begin{eqnarray*}
\inf_{g(x) \in {C}} f(x) &\le& f(\overline{x}) + c_*\, \|\overline{u}\|.
\end{eqnarray*}
\end{enumerate}
When these equivalent properties hold$,$ one has for all $c > c_*$ that
\begin{eqnarray*}
\mathrm{argmin}_{g(x) \in C} f(x) &=& \mathrm{argmin}_{\mathbb{R}^n} \big\{ f(x) + c\, \mathrm{dist}(g(x), C)  \big\}.
\end{eqnarray*}
\end{theorem}

\begin{proof}
It suffices to show the equivalence of (i) with (ii) since the last conclusion is a direct consequence of Theorem~\ref{DL31}.

(i) $\Rightarrow$ (ii) 
Assume that
\begin{eqnarray*}
\inf_{g(x) \in {C}} f(x) &=& \inf_{x \in \mathbb{R}^n} \big\{ f(x) + c_*\, \mathrm{dist} \big(g(x), {C}\big) \big\}
\end{eqnarray*}
for some $c_* > 0.$ Let $(\overline{x}, \overline{u}) \in \mathbb{R}^n \times \mathbb{R}^m$ be such that $g(\overline{x}) \in {C} + \overline{u}.$
Then
\begin{eqnarray*}
\inf_{g(x) \in {C}} f(x)
&\le& f(\overline{x}) + c_* \, \mathrm{dist} \big(g(\overline{x}), {C}\big) \\
&=& f(\overline{x}) + c_* \, \inf \big\{ \|g(\overline{x}) - y\| \mid y \in {C} \big\} \\
&\le& f(\overline{x}) + c_* \, \inf \big\{ \|g(\overline{x}) - (y + {\overline{u}})\| + \|\overline{u}\|  \mid y \in {C} \big\} \\
&=& f(\overline{x}) + c_* \, \mathrm{dist} \big(g(\overline{x}), {C} + {\overline{u}}\big) + c_*\, \|\overline{u}\| \\
&=& f(\overline{x}) + c_*\, \|\overline{u}\|,
\end{eqnarray*}
which gives the desired result.

(ii) $\Rightarrow$ (i) Let $\overline{x} \in \mathbb{R}^n$ and take any $\epsilon > 0.$ There exists a point $y \in {C}$ such that
\begin{eqnarray*}
\|g(\overline{x}) - y\| &\le& \mathrm{dist} \big(g(\overline{x}), {C}\big) + \epsilon.
\end{eqnarray*}
Let $\overline{u} := g(\overline{x}) - y.$ Then $g(\overline{x}) = y + \overline{u} \in {C} + \overline{u}.$ By hypothesis, we get
\begin{eqnarray*}
\inf_{g(x) \in {C}} f(x)
&\le& f(\overline{x}) + c_* \, \|\overline{u} \| \\
&\le& f(\overline{x}) + c_*\, \mathrm{dist} \big(g(\overline{x}), {C}\big) + c_* \epsilon.
\end{eqnarray*}
Taking the limit as $\epsilon \searrow 0$ yields the inequality
\begin{eqnarray*}
\inf_{g(x) \in {C}} f(x) &\le& f(\overline{x}) + c_*\, \mathrm{dist} \big(g(\overline{x}), {C}\big).
\end{eqnarray*}
Since $\overline{x}$ was arbitrary in $\mathbb{R}^n,$ it follows that
\begin{eqnarray*}
\inf_{g(x) \in {C}} f(x) &\le& \inf_{x \in \mathbb{R}^n} \big\{ f(x) + c_*\, \mathrm{dist} \big(g(x), {C}\big) \big\}.
\end{eqnarray*}
Since the opposite inequality holds trivially, the desired equality follows.
\end{proof}

To study the exactness of penalty functions, we can embed the problem~\eqref{Problem0} into a parametrized family of mathematical programs. For each $u \in \mathbb{R}^m$ consider the perturbed problem
\begin{equation*}
\textrm{minimize } f(x) \quad \textrm{ subject to } g(x) \in C + u.
\end{equation*}
Accordingly, we define the optimal value function $V \colon \mathbb{R}^m \to \mathbb{R} \cup \{\pm \infty\}$ by
\begin{eqnarray*}
V(u) &:=&
\begin{cases}
\inf_{g(x) \in C + u} f(x) & \textrm{ if }  \{x \in \mathbb{R}^n \mid g(x) \in C + u \} \ne \emptyset, \\ 
\infty & \textrm{ otherwise.}
\end{cases}
\end{eqnarray*}

\begin{corollary}
If there exists a constant $c_* > 0$ such that 
\begin{eqnarray*}
\inf_{g(x) \in {C}} f(x) &=& \inf_{x \in \mathbb{R}^n} \big\{ f(x) + c_*\, \mathrm{dist} \big(g(x), {C}\big) \big\},
\end{eqnarray*}
then 
\begin{eqnarray*}
\liminf_{u \to 0} \frac{V(u) - V(0)}{\|u  \|} &>& -\infty.
\end{eqnarray*}
The converse holds true if $f$ is bounded from below on $\mathbb{R}^n.$
\end{corollary}

\begin{proof}
Assume that
\begin{eqnarray*}
\inf_{g(x) \in {C} } f(x) &=& \inf_{x \in \mathbb{R}^n} \big\{f(x) + c_*\, \mathrm{dist} \big(g(x), {C}\big) \big\}.
\end{eqnarray*}
Take any $u \in \mathbb{R}^m.$ By Theorems~\ref{DL31} and \ref{DL32}, if there exists a point $\overline{x} \in \mathbb{R}^n$ with $g(\overline{x})  \in {C} + u,$ then 
\begin{eqnarray*}
\inf_{g(x) \in {C}} f(x) &\le& f(\overline{x}) + c_*\, \|u \|,
\end{eqnarray*}
which easily yields that
\begin{eqnarray*}
V(0) &\le& V(u) + c_*\, \|u \|.
\end{eqnarray*}
This inequality is still valid when there is no point $\overline{x}$ with $g(\overline{x})  \in {C} + u.$ Since $u \in \mathbb{R}^m$ was arbitrary, we get the desired conclusion
\begin{eqnarray*}
\liminf_{u \to 0} \frac{V(u) - V(0)}{\|u \|} &\ge& -c_* \ > \ -\infty.
\end{eqnarray*}

Now assume that $f$ is bounded from below on $\mathbb{R}^n$ and the converse is not true. By Theorem~\ref{DL32}, there exists a sequence $(x_k, u_k) \in \mathbb{R}^n \times \mathbb{R}^m$ with $g(x_k) \in C + u_k$ such that
\begin{eqnarray*}
\inf_{g(x) \in {C}} f(x) &>& f(x_k) + k\, \|u_k\| \ \ge \ \inf_{x \in \mathbb{R}^n} f(x)  + k\, \|u_k\|.
\end{eqnarray*}
Certainly $u_k \ne 0$ and $u_k \to 0$ as $k \to \infty.$ Moreover, we have
\begin{eqnarray*}
0 &>& \big [f(x_k) -  \inf_{g(x) \in {C}} f(x)  \big ] + k\, \|u_k \| \ \ge \ V(u_k) - V(0) + k\, \|u_k \|.
\end{eqnarray*}
Hence
\begin{eqnarray*}
0 &>& \frac{V(u_k) - V(0)}{\|u_k\|} + k,
\end{eqnarray*}
which is impossible for $k$ sufficiently large.
\end{proof}

The converse of the corollary above is not true if $f$ is not bounded from below on $\mathbb{R}^n;$ see Example~\ref{VD52}(ii).

\begin{remark}{\rm
The results presented in this section can be applied to an abstract constraint of the form $x \in X$ with $X$ being a subset of $\mathbb{R}^n$ 
 by substituting $f$ by $f  + \delta_X(\cdot).$
}\end{remark}

\section{Exact penalty functions: the locally Lipschitz case} \label{Section4}

This section is inspired by the work of Zaslavski \cite{Zaslavski2005, Zaslavski2013}. Indeed, consider the problem~\eqref{Problem} with the objective function $f$ being locally Lipschitz continuous and the constraint set $S$ being given by 
\begin{eqnarray*}
S &:=& \{x \in \mathbb{R}^n \ | \ g(x) \in C\},
\end{eqnarray*}
where $g \colon \mathbb{R}^n \to \mathbb{R}^m$ is a locally Lipschitz mapping and $C$ is a nonempty closed subset of $\mathbb{R}^m.$
Assume that the optimal value $f_* := \inf_{g(x) \in C} f(x)$ is finite. Under regularity conditions, {\em local and at infinity}, we will give a necessary and sufficient condition 
for the following to hold
\begin{eqnarray*}
\inf_{g(x) \in C} f(x) &=& \inf_{x \in \mathbb{R}^n} \big\{ f(x) + c\, \mathrm{dist}(g(x), C) \big\}
\end{eqnarray*}
for all $c$ sufficiently large. 

Recall that for each vector $w \in \mathbb{R}^m,$ the function $\langle w, g\rangle \colon \mathbb{R}^n \to \mathbb{R}$ is defined by
\begin{eqnarray*}
\langle w, g\rangle (x) &:=& \langle w, g(x)\rangle \quad \textrm{ for } \quad x \in \mathbb{R}^n.
\end{eqnarray*}
Let us start with the following.
\begin{definition}{\rm
We say that the {\em (generalized) Mangasarian--Fromovitz constraint qualification} (MFCQ) holds at $x \in S$ if there is no unit vector $w \in N_C(g(x))$ such that $0 \in \partial \langle w, g \rangle (x).$
}\end{definition}

\begin{remark}{\rm
When $C = \{0\} \times (\mathbb{R}_-)^{m - p} \subset \mathbb{R}^m$ and $g$ is of class $C^1,$ the above definition reduces to the traditional one; see \cite{Mangasarian1967}.
}\end{remark}

For each $x  \in \mathbb{R}^n \setminus S$ we put
\begin{eqnarray*}
\nu(x) &:=& \inf \|\lambda u + (1 - \lambda) v\|
\end{eqnarray*}
with the infimum being taken over all real numbers $\lambda \in (0, 1)$ and vectors $u \in \partial f(x), v \in \partial \langle w, g \rangle (x),$ and $w \in \frac{g(x) - \Pi_{C}(g(x))}{\mathrm{dist}(g(x), C)}.$

Let $K_{\infty}(f, S)$ denote the set of all $t \in \mathbb{R}$ for which there exists a sequence $x_k \in \mathbb{R}^n \setminus S$ satisfying the following conditions:
\begin{eqnarray*}
\lim_{k \to \infty} \|x_k\| = \infty,  \quad \lim_{k \to \infty}  \|x_k\| \nu(x_k) = 0, \quad \lim_{k \to \infty} f(x_k) = t, \quad 
\textrm{ and } \quad \lim_{k \to \infty}  \mathrm{dist}(g(x_k), C) = 0.
\end{eqnarray*}

\begin{remark}{\rm
(i) The definition of $K_{\infty}(f, S)$  is strongly related to the {\em weak Palais--Smale condition} (see, for example, \cite{Kurdyka2000-1, Rabier1997}).

(ii) When $C = \{0\} \subset \mathbb{R}^{m}$ and $f$ and $g$ are of class $C^1,$ we have that
\begin{eqnarray*}
\nu(x) &=& \inf \big\{  \|\lambda \nabla f(x)  + (1 - \lambda) \nabla \|g\|(x) \| \mid \lambda \in (0, 1)\big\}.
\end{eqnarray*}

(iii) If for some $c_* > 0$ we have
\begin{eqnarray*}
-\infty \ < \ \inf_{g(x) \in C} f(x) &=& \inf_{x \in \mathbb{R}^n} \big\{ f(x) + c_*\, \mathrm{dist}(g(x), C)  \big\},
\end{eqnarray*}
then the function 
$$\mathbb{R}^n \rightarrow \mathbb{R} \cup \{\infty\}, \quad x \mapsto f(x) + c_*\, \mathrm{dist}(g(x), C) ,$$ must be bounded from below. Conversely, we have the following result, which improves \cite[Theorems~1.5~and~3.1]{Zaslavski2005} and \cite[Theorems~1.3]{Zaslavski2013} (in the finite-dimensional setting).
}\end{remark}

\begin{theorem} \label{DL41}
Let the following two assumptions hold
\begin{enumerate}[{\rm ({A}1)}]
\item the {\rm (MFCQ)} is satisfied at every global minimum of $f$ on $S;$

\item the inclusion $K_{\infty}(f, S) \subset (f_*, \infty)$ is valid.
\end{enumerate}
Then the following properties are equivalent:

\begin{enumerate}[{\rm (i)}]
\item for all $c$ sufficiently large$,$ one has
\begin{eqnarray*}
\inf_{g(x) \in C} f(x) &=& \inf_{x \in \mathbb{R}^n} \big\{ f(x) + c\, \mathrm{dist}(g(x), C)  \big\};
\end{eqnarray*}

\item there exists a constant $c_* > 0$ such that 
\begin{eqnarray*}
\inf_{x \in \mathbb{R}^n}\big\{ f(x) + c_* \mathrm{dist}(g(x), C)   \big\} &>& -\infty.
\end{eqnarray*}
\end{enumerate}
When these equivalent properties hold, one has for any $c > c_*$ that
\begin{eqnarray*}
\mathrm{argmin}_{g(x) \in C} f(x) &=& \mathrm{argmin}_{\mathbb{R}^n} \big\{ f(x) + c\, \mathrm{dist}(g(x), C)  \big\};
\end{eqnarray*}
moreover, if $\overline{x}$ is a global minimizer for the problem~\eqref{Problem0}, then there exists a vector $w \in N_C(g(\overline{x})) \cap \mathbb{B}$ satisfying the following optimality condition
\begin{eqnarray*}
0 &\in& \partial f(\overline{x}) + c\, \partial \langle w, g\rangle (\overline{x}).
\end{eqnarray*}
\end{theorem}

To prove the theorem, we need the following simple fact.

\begin{lemma} \label{BD41}
Let $x_k, v_k  \in \mathbb{R}^n$ and $w_k \in \mathbb{R}^m$ be sequences converging to $\overline{x}, \overline{v}$ and $\overline{w},$ respectively, such that $v_{k} \in \partial \langle w_k, g\rangle (x_k)$ for all $k.$ Then $\overline{v} \in \partial \langle \overline{w}, g\rangle (\overline{x}).$
\end{lemma}

\begin{proof}
Since $v_{k} \in \partial \langle w_k, g\rangle (x_k),$ it follows from Lemma~\ref{BD28} that $v_k \in D^* g(x_k)(w_k),$
which by definition is equivalent to $(v_k, -w_k) \in N_{\mathrm{gph} g}(x_k, g(x_k)).$ Letting $k \to \infty$ and applying Lemma~\ref{BD21}, we get $(\overline{v}, -\overline{w}) \in N_{\mathrm{gph} g}(\overline{x}, g(\overline{x})),$ and so $\overline{v} \in D^* g(\overline{x})(\overline{w}).$ By Lemma~\ref{BD28} again, $\overline{v} \in \partial \langle \overline{w}, g\rangle (\overline{x}),$ as required.
\end{proof}

\begin{proof}[Proof of Theorem~\ref{DL41}]
We first prove the equivalence of (i) with (ii). To see this, define the function $\psi \colon \mathbb{R}^n \to \mathbb{R}_+$ by
\begin{eqnarray*}
\psi(x) &:=& \mathrm{dist}(g(x), C) \quad \textrm{ for } \quad x \in \mathbb{R}^n,
\end{eqnarray*}
which is locally Lipschitz continuous and residual for $S.$ 
By Theorem~\ref{DL31}, it suffices to show the equivalence of (ii) with the following property:
\begin{enumerate}[{\rm ($\star$)}]{\em
\item there exists a constant $c_* > 0$ such that 
\begin{eqnarray*}
c_* \psi(x) &\ge& [f_* - f(x)]_+ \quad \textrm{ for all } \quad x \in \mathbb{R}^n.
\end{eqnarray*}
}\end{enumerate}
Indeed, the ($\star$) $\Rightarrow$ (ii) part is trivial. For the (ii) $\Rightarrow$ ($\star$) part we likewise proceed by contradiction.
Let $c_* > 0$ be such that 
\begin{eqnarray*}
\inf_{x \in \mathbb{R}^n}\left\{ f(x) + c_*\psi(x) \right\} &>& -\infty,
\end{eqnarray*}
and suppose that for each $k > c_*,$ there exists $x_k\in \mathbb{R}^n$ such that
\begin{eqnarray*}\label{equa-1}
k \psi(x_k) &<& [f_* - f(x_k)]_+.
\end{eqnarray*}
Since the function $\psi$ is nonnegative, $f_*-f(x_k) > 0.$ Moreover, if $\psi(x_k) = 0$ then $x_k \in S$ and so $f_* \le f(x_k),$ a contradiction. 
Therefore
\begin{eqnarray}\label{Eqn2}
0 &<& k \psi(x_k) \ < \ f_*-f(x_k).
\end{eqnarray}

Define the function $F_k \colon \mathbb{R}^n \to \mathbb{R}, x \mapsto F_k(x),$ by
\begin{eqnarray*}
F_k(x) &:=& f(x) + k \psi(x),
\end{eqnarray*}
which is locally Lipschitz continuous and satisfies
\begin{eqnarray*}
-\infty \ < \ \frak{m} :=  \inf_{x \in \mathbb{R}^n} \big\{ f(x) + c_*\psi(x) \big\} &\le& \inf_{x \in \mathbb{R}^n} F_k(x) \ < \ f_*,
\end{eqnarray*}
where the last inequality follows from \eqref{Eqn2}.

Put 
\begin{eqnarray*}
\varepsilon_k := f_* - \inf_{x \in \mathbb{R}^n} F_k(x)  \quad \textrm{ and } \quad \lambda_k := \tfrac{\|x_k\|+1}{2}.
\end{eqnarray*}
Then $0 < \varepsilon_k < f_* - \frak{m}$ and we deduce from \eqref{Eqn2} that
\begin{eqnarray*}
F_k(x_k) &<&  f_* \ = \ \inf_{x \in \mathbb{R}^n} F_{k}(x) + \varepsilon_k.
\end{eqnarray*}
By the Ekeland variational principle (see \cite{Ekeland1979}), there exists $x_k' \in \mathbb{R}^n$ having the following conditions:
\begin{enumerate}[({c}1)]
\item $F_k(x_k') \leq F_k(x_k) <  f_* = \inf_{x \in \mathbb{R}^n} F_k(x)+\varepsilon_k,$

\item $\|x_k' - x_k\|\leq \lambda_k,$ and

\item $F_k(x_k')\leq F_k(x)+\frac{\varepsilon_k}{\lambda_k}\|x - x_k'\|$ for all $x\in\mathbb{R}^n$.
\end{enumerate}

Note that $\psi (x_k') > 0$ because otherwise we would have $x_k' \in S,$ so that
\begin{eqnarray*}
F_k(x_k') &=& f(x_k') + k\psi(x_k') \ = \ f(x_k') \ \geq \ f_*
\end{eqnarray*}
in contradiction to the condition~(c1). Moreover, we have
\begin{eqnarray} \label{Eqn3}
\frak{m} &\le& f(x_k') + {c_*}\psi(x_k')  \ \le \ f(x_k') + k\psi(x_k')  \ < \ f_*.
\end{eqnarray}
It follows that
\begin{eqnarray*}
\frak{m} +  (k - c_*) \psi(x_k')  & \le & f_*.
\end{eqnarray*}
Consequently, $\psi(x_k')$ tends to $0$ as $k \to \infty.$ This when combined with \eqref{Eqn3} implies that the sequence $f(x_k')$ is bounded and has all its cluster points in the interval $[\frak{m}, f_*].$

From the condition~(c2), we easily deduce that
\begin{eqnarray} \label{Eqn4}
\frac{\|x_k\| - 1}{2} &\le& \|x_k'\| \ \le \ \frac{3\|x_k\| + 1}{2}.
\end{eqnarray}
In particular, the sequence $x_k$ is bounded if and only if the sequence $x_k'$ is bounded.

On the other hand, by the condition~(c3), $x_k'$ is a minimizer of the locally Lipschitz function
$$\mathbb{R}^n\to \mathbb{R}, \quad x\mapsto F_k(x)+\frac{\varepsilon_k}{\lambda_k}\|x - x_k'\|.$$
Lemma~\ref{FermatRule} (Fermat rule) when combined with Lemma~\ref{SumRule} (sum rule) yields
\begin{eqnarray}
0& \in & \partial \left[F_k +\frac{\varepsilon_k}{\lambda_k}\|\cdot - x_k'\|\right](x_k') \notag\\
& \subset & \partial  f(x_k') + k \partial \psi(x_k') + \frac{\varepsilon_k}{\lambda_k}\mathbb{B}. \notag
\end{eqnarray}
Note that $g(x_k') \not \in C$ because $\psi(x_k') > 0.$ Moreover, we know from Lemmas~\ref{BD22} and \ref{ChainRule} that
\begin{eqnarray*}
\partial \psi(x_k') &\subset& \bigcup_{w} \partial \langle w, g\rangle (x_k'),
\end{eqnarray*}
where the union is taken on all vectors $w$ belonging to the set $\partial \psi (g(x_k')),$ which is
\begin{eqnarray*}
\frac{g(x_k') - \Pi_{C}(g(x_k'))}{\psi(x_k')}.
\end{eqnarray*}
Therefore, we can find vectors $u_k \in \partial f(x_k')$ and $v_{k} \in \partial \langle w_k, g\rangle (x_k')$  
for some $w_k \in \partial \psi (g(x_k')) $ such that 
\begin{eqnarray}\label{Eqn5}
0 &\in& u_k + k v_k + \frac{\varepsilon_k}{\lambda_k} \mathbb{B}.
\end{eqnarray}
There are two cases to be considered.

\subsubsection*{Case 1: the sequence $\{x_k\}$ is bounded.} By \eqref{Eqn4}, the sequence $\{x_k'\}$ is bounded too.
Passing to a subsequence if necessary, we can suppose that $x_k'$ converges to some point $\overline{x} \in \mathbb{R}^n.$ Certainly, $\psi (\overline{x}) = 0$ (equivalently, $\overline{x} \in S$) and $f(\overline{x}) = f_*,$ which imply that $\overline{x}$ is a global minimizer of $f$ on $S.$ 

On the other hand, since the mappings $f$ and $g$ are Lipschitz continuous around $\overline{x}$ and $\|w_k\| = 1,$ it follows from Lemma~\ref{BD26}
that the sequences $u_k$ and $v_{k}$ must be bounded. So we may assume that the sequences $u_k, v_{k}$ and $w_k$ converge to vectors $\overline{u}, \overline{v}$ and $\overline{w},$ respectively. Certainly, $\overline{u} \in \partial f(\overline{x}),$ $\overline{v} \in \partial \langle \overline{w}, g \rangle (\overline{x})$ and $\overline{w} \in N_C(g(\overline{x}))$ with $\|\overline{w}\| = 1$ (due to Lemmas~\ref{BD21}, \ref{BD23} and \ref{BD41}). 

Observe that \eqref{Eqn5} can be rewritten as
\begin{eqnarray*}
0 &\in& \frac{1}{k} u_k + v_k + \frac{1}{k}  \frac{\varepsilon_k}{\lambda_k} \mathbb{B}.
\end{eqnarray*}
Equivalently,
\begin{eqnarray*}
\left\| \frac{1}{k} u_k + v_k \right \| & \le & \frac{1}{k}  \frac{\varepsilon_k}{\lambda_k} \ < \ 
\frac{2(f_* - \frak{m})}{k(\|x_k\| + 1)}.
\end{eqnarray*}
Letting $k \to \infty,$ one gets $\overline{v} = 0,$ in contradiction to (A1).

\subsubsection*{Case 2: the sequence $\{x_k\}$ is unbounded.} 
By \eqref{Eqn4}, the sequence $\{x_k'\}$ is also unbounded. Passing to a subsequence if necessary, we can suppose that $\|x_k'\| \to \infty$ as $k \to \infty.$ Note that \eqref{Eqn5} can be rewritten as
\begin{eqnarray*}
0 &\in& \frac{1}{{1 + k}} u_k + \frac{k}{{1 + k}} v_k + \frac{1}{{1 + k}} \cdot \frac{\varepsilon_k}{\lambda_k} \mathbb{B}.
\end{eqnarray*}
By definition, then
\begin{eqnarray*}
\nu(x_k') & \leq & 
\left \| \frac{1}{{1 + k}} u_k + \frac{k}{{1 + k}} v_k \right\| \ \le \
\frac{\varepsilon_k}{(1 + k)\lambda_k} \ < \ 
\frac{2(f_* - \frak{m})}{{(1 + k)} (\|x_k\| + 1)}.
\end{eqnarray*}
In combining this with \eqref{Eqn4}, we obtain
\begin{eqnarray*}
(\|x_k'\| + 1)\nu(x_k') & < & \frac{(f_* - \frak{m})}{3 ({1 + k})}.
\end{eqnarray*}
Thus, $\|x_k'\| \nu(x_k') \to 0$ as $k \to \infty.$ On the other hand, we have shown that $x_k' \not \in S,$ $\psi(x_k') \to 0$ as $k \to \infty$ 
and the sequence $f(x_k')$ has a cluster point, say, $t,$ which belongs to the interval $[\frak{m}, f_*].$ Certainly $t \in K_{\infty}(f, S),$ in contradiction to (A2). The equivalence of (ii) with ($\star$), and hence with (i), is proved.

Finally, assume that the property~(i) holds. By Theorem~\ref{DL31}, one has for any $c > c_*$ that
\begin{eqnarray*}
\mathrm{argmin}_{g(x) \in C} f(x) &=& \mathrm{argmin}_{\mathbb{R}^n} \big\{ f(x) + c\, \mathrm{dist}(g(x), C)  \big\}.
\end{eqnarray*}
Let $\overline{x}$ be a global minimum for the problem~\eqref{Problem0}. Then $\overline{x}$ is also a global minimum for the unconstrained optimization problem
\begin{equation*}
\textrm{minimize}_{x \in \mathbb{R}^n} \big\{f(x) + c\, \mathrm{dist}(g(x), C) \big\},
\end{equation*}
and so $0 \in \partial \big\{ f(\cdot) + c\, \mathrm{dist}(g(\cdot), C) \big\}(\overline{x})$ due to Lemma~\ref{FermatRule}.
This, together with Lemmas~\ref{BD22}, \ref{SumRule} and \ref{ChainRule}, yields
\begin{eqnarray*}
0 & \in & \partial f(\overline{x}) + c\, \partial \langle w, g\rangle (\overline{x})
\end{eqnarray*}
for some $w \in N_C(g(\overline{x})) \cap \mathbb{B},$ as required.
\end{proof}

The assumptions (A1) and (A2) cannot be omitted as shown below.

\begin{example}{\rm
(i) Let $n := 1, S := \{x \in \mathbb{R} \mid x^2 = 0\}$ and $f(x) := x^3.$ Clearly, for any $c > 0$ we have
$$\inf_{x \in S} f(x) = 0 > -\infty  = \inf_{x \in \mathbb{R}} \big\{ f(x) + c\, \psi(x) \big\},$$ 
where $\psi(x) := x^2.$ Observe that (A2) holds while (A1) does not.

(ii) Let $n := 2, S := \{x := (x_1, x_2) \in \mathbb{R}^2 \ | \ x_1 = 0\}$ and $f(x) := e^{x_1 x_2}.$ Clearly, for any $c > 0$ we have
$$\inf_{x \in S} f(x) = 1 > 0 = \inf_{x \in \mathbb{R}^2} \big\{ f(x) + c\, \psi(x) \big\},$$ 
where $\psi(x) := |x_1|.$ Observe that (A1) is satisfied, but (A2) is not.

(iii) Let $n := 2, S := \{x := (x_1, x_2) \in \mathbb{R}^2 \mid x_1 - x_2 = 0\}$ and $f(x) := x_1 - x_2.$ 
Observe that the assumption~(1.1) in \cite{Zaslavski2005} and the assumption~(1.6) in \cite{Zaslavski2013} do not satisfy, and so we can not apply the results in these papers. On the other hand, it is easy to see that the assumptions (A1) and (A2) hold, and indeed, for any $c \ge 1$ we have
$$\inf_{x \in S} f(x) = 0 = \inf_{x \in \mathbb{R}^2} \big\{ f(x) + c\, \psi(x) \big\},$$ 
where $\psi(x) := |x_1 - x_2|.$

(iv) Let $n := 2, S := \{x := (x_1, x_2) \in \mathbb{R}^2 \mid x_1 = 0\}$ and $f(x) := x_1^3 + x_2^2.$ For any $c > 0$ we have
$$\inf_{x \in S} f(x) = 0 > -\infty  = \inf_{x \in \mathbb{R}^2} \big\{ f(x) + c\, \psi(x) \big\},$$ 
where $\psi(x) := |x_1|.$ Observe that the assumptions (A1) and (A2) are satisfied.
}\end{example}

\begin{remark}{\rm
(i) Theorem~\ref{DL41} can be strengthened by replacing $f$ by $f_0 + \delta_X$ with $f_0  \colon \mathbb{R}^n \rightarrow \mathbb{R}$ being a locally Lipschitz function and $X$ being a closed subset of $\mathbb{R}^n.$
Since we do not use this fact, we leave the details for the interested reader.

(ii) The assumption~(A2) holds trivially when $K_{\infty}(f, S)$ is an empty set. Consequently, 
Theorem~\ref{DL41} improves \cite[Theorems~1.5~and~3.1]{Zaslavski2005} where it is assumed that $C = \{0\} \times \mathbb{R}_{-}^{m - p},$ $f$ and $g$ are locally Lipschitz, $f$ is {\em coercive} and the linear independence constraint qualification holds at every global minimum of $f$ on $S.$ In the finite-dimensional setting, Theorem~\ref{DL41} also improves \cite[Theorems~1.3]{Zaslavski2013}.
}\end{remark}

\section{Exact penalty functions: the semi-algebraic case} \label{Section5}

In this section we investigate exact penalty functions for constrained optimization problems with semi-algebraic data. The main tool for this study will be \L ojasiewicz inequalities on unbounded sets. Note that the results presented here are still valid for functions definable in a polynomially bounded o-minimal structure (see~\cite{Coste2000, Dries1998} for more details on the subject); however, to lighten the exposition, we do not pursue this extension here.

\subsection{\L ojasiewicz inequalities on unbounded sets}
In this subsection, we present some {\L}ojasiewicz inequalities on unbounded sets for (not necessarily continuous) semi-algebraic functions, and
for the convenience of the reader, we provide detailed proofs to make the paper self-contained. The results presented here, which are inspired by the work of Dinh, H\`a and Thao \cite{Dinh2012} (see also \cite{Denkowski2017-2, Kosiba2025, Lee2022, Loi2016}), will play a crucial role in establishing exact penalty properties in constrained optimization with semi-algebraic data. 

Let us start with the following definition.
\begin{definition}{\rm
Let $X$ be a subset of $\mathbb{R}^n$ and $\phi, \psi \colon X \rightarrow \mathbb{R}$ be functions. A sequence $\{x_k\}$ in $X$ is said to be
\begin{enumerate}[{\rm (i)}]
\item a {\em sequence of the first type} for the pair $(\phi, \psi)$ if $\psi(x_k) \to  0$ and there exists $\epsilon > 0$ such that $|\phi(x_k)| \ge \epsilon$ for all $k;$
\item a {\em sequence of the second type} for the pair $(\phi, \psi)$ if the sequence $\psi(x_k)$ is bounded and $|\phi(x_k)| \to  \infty.$
\end{enumerate}
}\end{definition}

The following result improves \cite[Proposition~3.7]{Dinh2012} in which $X = \mathbb{R}^n,$ $\psi$ is the absolute value function of a  polynomial and $\phi$ is the distance function $x \mapsto \mathrm{dist}(x, \psi^{-1}(0)).$

\begin{lemma}\label{BD51}
Let $X$ be a semi-algebraic subset of $\mathbb{R}^n$ and $\phi, \psi \colon X \rightarrow \mathbb{R}$ be nonnegative and semi-algebraic functions. The following two conditions are equivalent:
\begin{enumerate}[{\rm (i)}]
\item there exist some constants $c > 0, \epsilon > 0$ and $\alpha \in (0, 1]$ such that
\begin{eqnarray*}
c\, [\psi(x)]^{\alpha}  &\ge& \phi(x) \quad \textrm{ for all } \quad x \in \psi^{-1}([0, \epsilon));
\end{eqnarray*}

\item there is no sequence of the first type for the pair $(\phi, \psi).$
\end{enumerate}
\end{lemma}
\begin{proof}
(i) $\Rightarrow$ (ii) The implication is straightforward.

(ii) $\Rightarrow$ (i) Since there is no sequence of the first type for the pair $(\phi, \psi),$ we have
\begin{enumerate}[{\em Fact} 1.]
\item {\em $\phi(x) \to 0$ as $\psi(x) \to 0;$ in particular, $\phi(x)  = 0$ provided that $\psi(x) = 0.$}
\end{enumerate}

On the other hand, in light of Theorem~\ref{TarskiSeidenbergTheorem}, $\psi(X)$ is a semi-algebraic set in $\mathbb{R},$ so it is a finite union of points and intervals (see Lemma~\ref{FiniteLemma}). Consequently, there exists a constant $\epsilon > 0$ such that either $(0, \epsilon) \cap \psi(X) = \emptyset$ or $(0, \epsilon) \subset \psi(X).$
In the first case, for all $x \in \psi^{-1}[0, \epsilon)$ we have $\psi(x) = 0,$ and by Fact~1, $\phi(x) = 0;$ consequently, (i) holds trivially for any choice of  $c > 0$ and $\alpha > 0.$ So, assume that the second case holds, i.e., $(0, \epsilon) \subset \psi(X).$ Then for all $t \in (0, \epsilon),$ the set $\psi^{-1}(t)$ is nonempty and so the function
$$\mu \colon (0, \epsilon) \to \mathbb{R} \cup \{\infty\}, \quad t\mapsto \sup_{x\in\psi^{-1}(t)}\phi(x),$$
is well defined. Observe that, by Theorem~\ref{TarskiSeidenbergTheorem}, the function $\mu$ is semi-algebraic.
Moreover, it follows from Fact~1 that $\mu(t) \to 0$ as $t \to 0^+.$ By Lemma~\ref{MonotonicityLemma} and decreasing $\epsilon$ if necessary, we can see that the function $\mu$ is finite and either constant or strictly monotone. There are two cases to consider.

\subsubsection*{Case 1.1:} $\mu$ is constant.

Then $\mu(t) = 0$ for all $t \in (0, \epsilon).$ It follows that for all $x \in \psi^{-1}(0, \epsilon),$
$$\psi(x) \geq 0 = \phi(x).$$
Then (i) holds trivially for any choice of $c > 0$ and $\alpha > 0.$

\subsubsection*{Case 1.2:} $\mu$ is not constant.

Thanks to Lemma~\ref{GrowthDichotomyLemma}(i), we can write
$$\mu(t) = a\, t^{\alpha} + o(t^{\alpha}) \quad \textrm{ as } \quad t \to 0^+$$
for some constants $a > 0$ and $\alpha \in \mathbb{Q}.$ Observe that $\alpha > 0$ because $\mu(t) \to 0$ as $t \to 0^+.$
Reducing $\epsilon$ and replacing $\alpha$ by $\min\{\alpha, 1\}$ (if necessary), we can see that
$$\mu(t)\leq 2 a\, t^{\alpha} \quad \textrm{ for } \quad t \in (0, \epsilon).$$
Hence (i) holds when $c := 2 a > 0.$
\end{proof}

The following corollary improves \cite[Lemma~3.2]{Lee2022} in which $X$ and $\phi$ are assumed to be compact and continuous, respectively;
see also \cite[Theorem~2.10]{Kosiba2025}.

\begin{corollary} \label{HQ51}
Let $X$ be a semi-algebraic subset of $\mathbb{R}^n$ and $\phi, \psi \colon X \rightarrow \mathbb{R}$ be nonnegative and semi-algebraic functions.
If in addition $\phi$ is bounded from above, then the following two conditions are equivalent:
\begin{enumerate}[{\rm (i)}]
\item there exist some constants $c > 0$ and $\alpha \in (0, 1]$ such that
\begin{eqnarray*}
c\, [\psi(x)]^{\alpha} &\ge& \phi(x) \quad \textrm{ for all } \quad x \in X;
\end{eqnarray*}

\item there is no sequence of the first type for the pair $(\phi, \psi).$
\end{enumerate}
\end{corollary}

\begin{proof}
(i) $\Rightarrow$ (ii) The implication is straightforward.

(ii) $\Rightarrow$ (i) By Lemma~\ref{BD51}, we can find constants $c_1 > 0, \epsilon > 0$ and $\alpha \in (0, 1]$ such that
\begin{eqnarray*}
c_1[\psi(x)]^\alpha &\ge& \phi(x) \quad \textrm{ for all } \quad x \in \psi^{-1}([0, \epsilon)).
\end{eqnarray*}
On the other hand, since $\phi$ is bounded from above, there exists a constant $M > 0$ such that $\phi(x) \le M$ for all $x \in X.$ Hence for all $x \in \psi^{-1}([\epsilon, \infty))$ we have
\begin{eqnarray*}
[\psi(x)]^{\alpha}
& \ge & \epsilon^\alpha \ = \ \frac{\epsilon^\alpha}{M}M 
\ \ge \ \frac{\epsilon^\alpha}{M}\, \phi(x).
\end{eqnarray*}
Letting $c := \max\{c_1, \frac{M}{\epsilon^\alpha}\},$ we get the desired conclusion.
\end{proof}
Note here that the assumption of boundedness of the function $\phi$ cannot be omitted; see Example~\ref{VD51}.

\begin{corollary}\label{HQ52}
Let $A, B$ be closed semi-algebraic subsets of $\mathbb{R}^n$ such that $A \cap B \ne \emptyset.$ Then there exist some constants $c > 0$ and $\alpha \in (0, 1]$ such that
\begin{eqnarray*}
c\, \big(\mathrm{dist}(x, A) + \mathrm{dist}(x, B) \big)^\alpha &\ge& \frac{\mathrm{dist}(x, A \cap B)}{1 + \|x\|^2} \quad \textrm{ for all } \quad x \in \mathbb{R}^n.
\end{eqnarray*}
\end{corollary}
\begin{proof}
Define the functions $\phi, \psi \colon \mathbb{R}^n \to \mathbb{R}$ by
\begin{eqnarray*}
\phi(x) := \frac{\mathrm{dist}(x, A \cap B)}{1 + \|x\|^2} \quad & \textrm{ and } & \quad 
\psi(x) := \mathrm{dist}(x, A) + \mathrm{dist}(x, B),
\end{eqnarray*}
which are nonnegative, continuous and semi-algebraic (by Theorem~\ref{TarskiSeidenbergTheorem}). Note that $\phi(x) = 0$ when $\psi(x) = 0$ and that $\phi(x) \to 0$ as $\|x\| \to \infty.$ Thus, the function $\phi$ is bounded from above and there is no sequence of the first type for the pair $(\phi, \psi).$ The desired conclusion follows directly from Corollary~\ref{HQ51}.
\end{proof}

\begin{lemma}\label{BD52}
Let $X$ be a semi-algebraic subset of $\mathbb{R}^n$ and $\phi, \psi \colon X \rightarrow \mathbb{R}$ be nonnegative and semi-algebraic functions. The following two conditions are equivalent:
\begin{enumerate}[{\rm (i)}]
\item there exist some constants $c > 0, R > 0$ and $\beta \ge 1$ such that
\begin{eqnarray*}
c\, [\psi(x)]^{\beta} &\ge& \phi(x) \quad \textrm{ for all } \quad x \in \psi^{-1}((R, \infty));
\end{eqnarray*}

\item there is no sequence of the second type for the pair $(\phi, \psi).$
\end{enumerate}
\end{lemma}

\begin{proof}
(i) $\Rightarrow$ (ii) The implication is straightforward.

(ii) $\Rightarrow$ (i) On the semi-algebraic set $\{ x \in X \mid \phi(x) > 0 \textrm{ and } \psi(x) > 0\}$ the reciprocal functions $\frac{1}{\phi}, \frac{1}{\psi}$ are well defined, nonnegative and semi-algebraic. Moreover, there is no sequence of the first type for the pair $(\frac{1}{\psi}, \frac{1}{\phi}).$ In view of Lemma~\ref{BD51}, we can find constants $c_1 > 0, \epsilon > 0$ and $\alpha \in (0, 1]$ such that
\begin{eqnarray*}
c_1\, \left[\frac{1}{\phi(x)} \right]^{\alpha} &\ge& \frac{1}{\psi(x)} \quad \textrm{ whenever } \quad \frac{1}{\psi(x)} < \epsilon.
\end{eqnarray*}
Hence (i) holds when $c := (c_1)^{\frac{1}{\alpha}}, R := \frac{1}{\epsilon}$ and $\beta := \frac{1}{\alpha}.$
\end{proof}

The following result improves \cite[Proposition~3.10]{Dinh2012} in which $X = \mathbb{R}^n,$ $\psi$ is the absolute value function of a  polynomial and $\phi$ is the distance function $x \mapsto \mathrm{dist}(x, \psi^{-1}(0)).$

\begin{proposition}\label{MD51}
Let $X$ be a semi-algebraic subset of $\mathbb{R}^n$ and $\phi, \psi \colon X \rightarrow \mathbb{R}$ be nonnegative and semi-algebraic functions. The following two conditions are equivalent:
\begin{enumerate}[{\rm (i)}]
\item there exist some constants $c > 0, \alpha > 0,$ and $\beta > 0$ with $\alpha \le 1 \le \beta$ such that
\begin{eqnarray*}
c\left( [\psi(x)]^{\alpha}  + [\psi(x)]^{\beta} \right)   &\ge& \phi(x) \quad \textrm{ for all } \quad x \in X;
\end{eqnarray*}

\item there are no sequences of the first and second types for the pair $(\phi, \psi).$
\end{enumerate}
\end{proposition}

\begin{proof}
(i) $\Rightarrow$ (ii) The implication is straightforward.

(ii) $\Rightarrow$ (i) By Lemma~\ref{BD51}, there exist some constants $c_1 > 0, \epsilon > 0$ and $\alpha \in (0, 1]$ such that
\begin{eqnarray*}
c_1\, [\psi(x)]^{\alpha}  &\ge& \phi(x) \quad \textrm{ for all } \quad x \in \psi^{-1}([0, \epsilon)).
\end{eqnarray*}
In view of Lemma~\ref{BD52}, there exist some constants $c_2 > 0, R > \epsilon$ and $\beta \ge 1$ such that
\begin{eqnarray*}
c_2\, [\psi(x)]^{\beta} &\ge& \phi(x) \quad \textrm{ for all } \quad x \in \psi^{-1}((R, \infty)).
\end{eqnarray*}
On the other hand, it follows from Fact~2 in the proof of Lemma~\ref{BD52} that there exists a constant $M > 0$ such that for all $x \in \psi^{-1}([\epsilon, R])$ we have $\phi(x) \le M,$ and hence 
\begin{eqnarray*}
[\psi(x)]^{\alpha} + [\psi(x)]^{\beta} 
& \ge & \epsilon^\alpha + \epsilon^\beta \ = \ \frac{\epsilon^\alpha + \epsilon^\beta}{M}M 
\ \ge \ \frac{\epsilon^\alpha + \epsilon^\beta}{M}\, \phi(x).
\end{eqnarray*}
Letting $c := \max\{c_1, c_2, \frac{M}{\epsilon^\alpha + \epsilon^\beta}\},$ we get the desired conclusion.
\end{proof}

The following simple example shows that the exponents $\alpha$ and $\beta$ in Proposition~\ref{MD51} are different in general.
\begin{example}\label{VD51}{\rm 
Consider the polynomial functions $\phi(x_1, x_2) := x_1^2 + x_2^2$ and $\psi(x_1, x_2) := x_1^2 + x_2^4$ on $\mathbb{R}^2.$ 
It is not hard to see that there are no constants $c > 0$ and $\alpha > 0$ such that
\begin{eqnarray*}
c\,[\psi(x_1, x_2)]^{\alpha} &\geq& \phi(x_1, x_2) \quad \textrm{ for all } \quad (x_1, x_2) \in \mathbb{R}^2.
\end{eqnarray*}
On the other hand, it holds that
\begin{eqnarray*}
[\psi(x_1, x_2)]^{\frac{1}{2}} + \psi(x_1, x_2) &\geq& \phi(x_1, x_2) \quad \textrm{ for all } \quad (x_1, x_2) \in \mathbb{R}^2.
\end{eqnarray*}
}\end{example}

We close this subsection by presenting a necessary and sufficient condition for a semi-algebraic function to have a global H\"olderian error bound; this result improves \cite[Theorem~3.4]{HaHV2017}, where the continuity of the function in question is assumed.

\begin{corollary}
Let $f \colon \mathbb{R}^n \rightarrow \mathbb{R}$ be a semi-algebraic function, and $S := \{x \in \mathbb{R}^n \mid f(x) \le 0\} \ne \emptyset.$ The following two conditions are equivalent:
\begin{enumerate}[{\rm (i)}]
\item there exist some constants $c > 0, \alpha > 0,$ and $\beta > 0$ with $\alpha \le 1 \le \beta$ such that
\begin{eqnarray*}
c\left( [f(x)]_+^{\alpha}  + [f(x)]_+^{\beta} \right)   &\ge& \mathrm{dist}(x, S), \quad \textrm{ for all } \quad x \in \mathbb{R}^n;
\end{eqnarray*}

\item for any sequence of points $x_k$ in $\mathbb{R}^n$ one has
\begin{enumerate}[{\rm ({ii}1)}]
\item if $f(x_k) \to 0,$ then $\mathrm{dist}(x, S) \to 0;$
\item if $\mathrm{dist}(x, S) \to \infty,$ then $f(x_k) \to \infty.$
\end{enumerate}
\end{enumerate}
\end{corollary}
\begin{proof}
This applies Proposition~\ref{MD51} to the semi-algebraic  functions $\psi(x) := [f(x)]_+$ and $\phi(x) := \mathrm{dist}(x, S)$ for $x \in \mathbb{R}^n.$
\end{proof}

\subsection{Necessary and sufficient conditions for exact penalization}

The results presented in this subsection are inspired by the works of Warga~\cite[Theorem~1]{Warga1992}, Dedieu \cite[Theorem~3.1]{Dedieu1992} and Luo, Pang and Ralph \cite[Chapter~2]{Luo1996}. Indeed, consider the problem~\eqref{Problem} with $S$ being a nonempty {\em semi-algebraic} subset of $\mathbb{R}^n$ and $f \colon \mathbb{R}^n \rightarrow \overline{\mathbb{R}}$ being a {\em semi-algebraic} function. Assume that the optimal value $f_* := \inf_S f$ is finite and let $\psi \colon \mathbb{R}^n \to \mathbb{R}_+$ be a residual function for $S$ which is {\em semi-algebraic.} Let's start with the following.

\begin{remark}{\rm
If for some $c_* > 0$ we have
\begin{eqnarray*}
\inf_{x \in S} f(x) &=& \inf_{x \in \mathbb{R}^n} \big\{ f(x) + c_*\, \psi(x) \big\},
\end{eqnarray*}
then by Theorem~\ref{DL31}, for any sequence of points $x_k$ in $\mathbb{R}^n$ with $\psi(x_k) \to 0$ the sequence $[f_* - f(x_k)]_+$ must converge to $0.$ The converse does not hold in general; see Example~\ref{VD52}(ii) below. On the other hand, we have the following result.
}\end{remark}

\begin{theorem} \label{DL51}
Assume that for any sequence of points $x_k$ in $\mathbb{R}^n$ with $\psi(x_k) \to 0$ one has $[f_* - f(x_k)]_+ \to 0.$ Then there exist constants $c_*  > 0$ and $\alpha \in (0, 1]$ such that for all $c > c_*,$
\begin{eqnarray*}
\inf_{x \in S} f(x) &=& \inf_{x \in \mathbb{R}^n} \big\{f(x) + c\, \big(1 + [f(x)]^2 \big)[\psi(x)]^\alpha\big\}, \\
\mathrm{argmin}_S f(x) &=& \mathrm{argmin}_{\mathbb{R}^n} \big\{f(x) + c\, \big(1 + [f(x)]^2 \big)[\psi(x)]^\alpha\big\}.
\end{eqnarray*}
\end{theorem}

\begin{proof}
Let $A := (-\infty, f_*] \times \{0\}$ and $B := \mathrm{cl} (f, \psi)(\mathrm{dom} f).$ Then $A$ and  $B$ are closed semi-algebraic sets in $\mathbb{R}^2.$ Moreover, our assumption yields $A \cap B = \{(f_*, 0)\}.$ In view of Corollary~\ref{HQ52}, there are some constants $c_* > 0$ and $\alpha \in (0, 1]$ such that
\begin{eqnarray*}
c_* \big(\mathrm{dist}(y, A) + \mathrm{dist}(y, B) \big)^\alpha &\ge& \frac{\mathrm{dist}(y, A \cap B)}{1 + \|y\|^2} \quad \textrm{ for all } \quad y \in \mathbb{R}^2.
\end{eqnarray*}
Take $x \in \mathbb{R}^n$ with $f(x) < f_*$ and let $y := (f(x), 0).$ We have
\begin{eqnarray*}
\mathrm{dist}(y, A) &=& 0, \\
\mathrm{dist}(y, B) &\le & \|(f(x), 0) - (f(x), \psi(x))\| = \psi(x), \\
\mathrm{dist}(y, A \cap B) &=& \|(f(x), 0) - (f_*, 0)\| = f_* - f(x).
\end{eqnarray*}
Hence
\begin{eqnarray*}
c_* [\psi(x)]^\alpha &\ge& \frac{f_* - f(x)}{1 + [f(x)]^2},
\end{eqnarray*}
or equivalently,
\begin{eqnarray*}
c_*{\big(1 + [f(x)]^2 \big)} [\psi(x)]^\alpha &\ge& {f_* - f(x)}.
\end{eqnarray*}
Clearly, this inequality is true for any $x \in \mathbb{R}^n$ with $f(x) \ge f_*.$ By Theorem~\ref{DL31}, we get the desired conclusion.
\end{proof}

\begin{remark}{\rm
(i) Theorem~\ref{DL51} can fail when ``semi-algebraic" is replaced by ``subanalytic'' (cf. \cite{Bierstone1988} for this notion); see Example~\ref{VD53}.

(ii) In general, the converse of Theorem~\ref{DL51} does not hold and the term $(1 + [f(x)]^2)$ cannot be omitted; see the next example.
}\end{remark}

\begin{example}\label{VD52}{\rm
(i) Let $n := 1, S := \{ 0\}$ and consider the semi-algebraic functions
\begin{eqnarray*}
f(x) &:=& x \quad \textrm{ and } \quad \psi(x) \ :=\ \frac{|x|} {1 + x^2}.
\end{eqnarray*}
Then $\psi$ is a residual function for $S$ and it holds that
\begin{eqnarray*}
\inf_{x \in S} f(x) &=& \inf_{x \in \mathbb{R}} \big\{ f(x) + (1 + [f(x)]^2) [\psi(x)] \big\},
\end{eqnarray*}
while
\begin{eqnarray*}
\psi(x) \to 0 \quad \textrm{ and } \quad [f_* - f(x)]_+ \to \infty \quad \textrm{ as } \ x \to -\infty.
\end{eqnarray*}

(ii) Let $n := 1, S := \{ 0\}$ and $f(x) := x.$ Consider the semi-algebraic function $\psi \colon \mathbb{R} \to \mathbb{R}_+$ defined by
\begin{eqnarray*}
\psi(x) &:=& 
\begin{cases}
|x| & \textrm{ if } |x| \le 1, \\
1 & \textrm{ otherwise,}
\end{cases}
\end{eqnarray*}
which is a residual function for $S.$ A direct calculation shows that 
\begin{eqnarray*}
\inf_{x \in S} f(x) &=& \inf_{x \in \mathbb{R}} \big\{ f(x) + (1 + [f(x)]^2) [\psi(x)] \big\}, 
\end{eqnarray*}
while for any $c > 0$ and any $\alpha > 0,$
\begin{eqnarray*}
\inf_{x \in S} f(x) &=& 0 \ > \ - \infty \ = \ \inf_{x \in \mathbb{R}} \big\{f(x) + c\, [\psi(x)]^\alpha \big\}.
\end{eqnarray*}
}\end{example}

The following result shows that the term $(1 + [f(x)]^2)$ can be removed provided that the function $f$ is bounded from below on $\mathbb{R}^n.$
\begin{theorem} \label{DL52}
If $f$ is bounded from below on $\mathbb{R}^n,$ then the following properties are equivalent:
\begin{enumerate}[{\rm (i)}]
\item there exist constants $c_* > 0$ and $\alpha \in (0, 1]$ such that for all $c > c_*,$
\begin{eqnarray*}
\inf_{x \in S} f(x) &=& \inf_{x \in \mathbb{R}^n} \big\{ f(x) + c\, [\psi(x)]^{\alpha} \big\};
\end{eqnarray*}

\item there exist constants $c_* > 0$ and $\alpha \in (0, 1]$ such that
\begin{eqnarray*}
c_* [\psi(x)]^{\alpha} &\ge& [f_* - f(x)]_+ \quad \textrm{ for all } \quad x \in \mathbb{R}^n;
\end{eqnarray*}

\item for any sequence of points $x_k$ in $\mathbb{R}^n$ with $\psi(x_k) \to 0,$ one has $[f_* - f(x_k)]_+ \to 0.$
\end{enumerate}
When these equivalent properties hold, one has moreover for all $c > c_*$ that
\begin{eqnarray*}
\mathrm{argmin}_S f(x) &=& \mathrm{argmin}_{\mathbb{R}^n} \big\{ f(x) + c\, [\psi(x)]^{\alpha} \big\}.
\end{eqnarray*}
\end{theorem}

\begin{proof}
The equivalence of (i) and (ii) as well as the last conclusion follow from Theorem~\ref{DL31}.
To get the equivalence of (ii) and (iii), we can apply Corollary~\ref{HQ51} to the functions $\phi$ and $\psi$ with $\phi \colon \mathbb{R}^n \to \mathbb{R}_+$ being defined by
\begin{eqnarray*}
\phi(x) &:=& [f_* - f(x)]_+ \quad \textrm{ for } \quad x \in \mathbb{R}^n,
\end{eqnarray*}
which is bounded from above on $\mathbb{R}^n.$
\end{proof}

\begin{remark}{\rm
(i) In general, the assumption in Theorem~\ref{DL52} that $f$ is bounded from below on $\mathbb{R}^n$ cannot be dropped; see Example~\ref{VD52}(ii). On the other hand, we have Theorem~\ref{DL53} below.

(ii) In the semi-algebraic setting, Theorem~\ref{DL52} directly implies \cite[Theorem~1]{Warga1992} and \cite[Theorem~3.1]{Dedieu1992}. To see this, let $f := f_0 + \delta_X,$ $f_0 \colon X \to \mathbb{R}$ be continuous and semi-algebraic, $X \subset \mathbb{R}^n$ and $S \subset X$ nonempty, compact and semi-algebraic. Also, let $\psi \colon \mathbb{R}^n \to \mathbb{R}_+$ be a residual function for $S$ which is continuous and semi-algebraic. Then it is not hard to see that $f$ is bounded from below on $\mathbb{R}^n$ and that the condition~(iii) of Theorem~\ref{DL52} is satisfied, and so there exist constants $c_* > 0$ and $\alpha \in (0, 1]$ such that for all $c > c_*,$
\begin{eqnarray*}
\inf_{x \in S} f_0(x) &=& \inf_{x \in X} \big\{ f_0(x) + c\, [\psi(x)]^{\alpha} \big\}, \\
\mathrm{argmin}_S f_0(x) &=& \mathrm{argmin}_{X} \big\{ f_0(x) + c\, [\psi(x)]^{\alpha} \big\}.
\end{eqnarray*}

It is worth noting here that, in the subanalytic setting, these relations are still valid (see \cite[Theorem~3.1]{Dedieu1992}) but can fail when $X$ is unbounded. Also, the equivalence of (iii) with (ii) (and hence with (i)) can fail when ``semi-algebraic" is replaced by ``subanalytic''; see the next example.
}\end{remark}

\begin{example}\label{VD53} {\rm
Let $n := 1,$ $S := \{0\}$ and consider the continuous subanalytic functions
\begin{eqnarray*}
f(x) &:=& 
\begin{cases}
1 - \frac{1}{x} & \textrm{ if } x \ge 1, \\
-x + 1 & \textrm{ if } 0 \le x < 1, \\
1 & \textrm{ otherwise,}
\end{cases}
\end{eqnarray*}
and
\begin{eqnarray*}
\psi(x) &:=& 
\begin{cases}
e^{-x + 1} & \textrm{ if } x \ge 1, \\
|x| & \textrm{ otherwise.}
\end{cases}
\end{eqnarray*}
Certainly $f_* := \inf_{x \in S} f(x) = 1$ and $\psi$ is a residual function for $S.$ Moreover, for any sequence of points $x_k$ in $\mathbb{R}$ with $\psi(x_k) \to 0,$ one has $[f_* - f(x_k)]_+ \to 0.$ However, the equivalent properties~(i) and (ii) in Theorem~\ref{DL52} fail to hold. This example also shows that \cite[Theorem~1]{Warga1992} and \cite[Theorem~3.1]{Dedieu1992} can fail for optimization problems over unbounded sets.
}\end{example}

\begin{theorem} \label{DL53}
The following properties are equivalent:
\begin{enumerate}[{\rm (i)}]
\item there exist constants $c_* > 0, \alpha > 0,$ and $\beta > 0$ with $\alpha \le 1 \le \beta$ such that for all $c > c_*,$
\begin{eqnarray*}
\inf_{x \in S} f(x) &=&  \inf_{x \in \mathbb{R}^n} \left\{ f(x) + c\left( [\psi(x)]^{\alpha}  + [\psi(x)]^{\beta} \right)\right\};
\end{eqnarray*}

\item there exist constants $c_* > 0, \alpha > 0,$ and $\beta > 0$ with $\alpha \le 1 \le \beta$ such that
\begin{eqnarray*}
c_*\left( [\psi(x)]^{\alpha}  + [\psi(x)]^{\beta} \right)   &\ge& [f_* - f(x)]_+ \quad \textrm{ for all } \quad x \in \mathbb{R}^n;
\end{eqnarray*}

\item for any sequence of points $x_k$ in $\mathbb{R}^n$ one has
\begin{enumerate}[{\rm ({iii}1)}]
\item if $\psi(x_k) \to 0,$ then $[f_* - f(x_k)]_+ \to 0;$
\item if $f(x_k) \to -\infty,$ then $\psi(x_k) \to \infty.$
\end{enumerate}
\end{enumerate}
When these equivalent properties hold, one has moreover for all $c > c_*$ that
\begin{eqnarray*}
\mathrm{argmin}_S f(x) &=& \mathrm{argmin}_{\mathbb{R}^n} \left\{ f(x) + c\left( [\psi(x)]^{\alpha}  + [\psi(x)]^{\beta} \right)\right\}.
\end{eqnarray*}
\end{theorem}

\begin{proof}
The equivalence of (i) and (ii) as well as the last conclusion follow from Theorem~\ref{DL31}, while the equivalence of (ii) and (iii) is known from Proposition~\ref{MD51}.
\end{proof}

\begin{remark}{\rm
(i) The exponents $\alpha$ and $\beta$ can take, respectively, to be $\frac{1}{N}$ and $N$ for some integer $N > 0.$

(ii) To incorporate an abstract constraint of the form $x \in X$ with $X$ being a semi-algebraic subset of $\mathbb{R}^n,$ we simply replace $f$ by $f  + \delta_X.$

(iii) The term $[\psi(x)]^{\beta}$ cannot be omitted as shown in the following example.
}\end{remark}

\begin{example}{\rm 
Let $n := 2$ and consider the polynomial functions $f(x_1, x_2) := -x_1^2 - x_2^2$ and $\psi(x_1, x_2) := x_1^2 + x_2^4.$ By definition, we have that 
$\psi$ is a residual function for $S := \{(0, 0)\}$ and $f_* := \inf_{(x_1, x_2) \in S} f(x_1, x_2) = 0.$ Moreover, it is not hard to see that there are no constants $c > 0$ and $\alpha > 0$ such that
\begin{eqnarray*}
c[\psi(x_1, x_2)]^{\alpha} &\geq&  [f_* - f(x_1, x_2)]_+  \quad \textrm{ for all } \quad (x_1, x_2) \in \mathbb{R}^2.
\end{eqnarray*}
On the other hand, it holds that
\begin{eqnarray*}
[\psi(x_1, x_2)]^{\frac{1}{2}} + \psi(x_1, x_2) &\geq& [f_* - f(x_1, x_2)]_+ \quad \textrm{ for all } \quad (x_1, x_2) \in \mathbb{R}^2,
\end{eqnarray*}
which, together with Theorem~\ref{DL53}, implies that for all $c > 1,$
\begin{eqnarray*}
\inf_{(x_1, x_2) \in S} f(x_1, x_2) &=&  \inf_{(x_1, x_2) \in \mathbb{R}^2} \left\{ f(x_1, x_2) + c\left( [\psi(x_1, x_2)]^{\frac{1}{2}} + \psi(x_1, x_2) \right)\right\}.
\end{eqnarray*}
}\end{example}

In the rest of this subsection, we consider the problem~\eqref{Problem} with $f := f_0 + \delta_X,$ $f_0 \colon X \to \mathbb{R}$ (not necessarily semi-algebraic) Lipschitz, $X \subset \mathbb{R}^n$ and $S \subset X$ nonempty and semi-algebraic. Let $\psi \colon \mathbb{R}^n \to \mathbb{R}_+$ be a residual function for $S$ which is {\em semi-algebraic.}

\begin{theorem} \label{DL54}
Assume that for any sequence of points $x_k$ in $X,$
\begin{enumerate}[{\rm ({A}1)}]
\item if $\psi(x_k) \to 0$ with $x_k \in X,$ then $\mathrm{dist}(x_k, S) \to 0;$
\item the distance function $\mathrm{dist}(\cdot, S)$ is bounded from above on $X.$
\end{enumerate}
Then there exist constants $c_* > 0$ and $\alpha \in (0, 1]$ such that for all $c > c_*,$
\begin{eqnarray*}
\inf_{x \in S} f_0 (x) &=& \inf_{x \in X} \big\{ f_0(x) + c\, [\psi(x)]^{\alpha} \big\}, \\
\mathrm{argmin}_S f_0(x) &=& \mathrm{argmin}_{X} \big\{ f_0(x) + c\, [\psi(x)]^{\alpha} \big\}.
\end{eqnarray*}
\end{theorem}

\begin{proof}
By Theorem~\ref{DL31}, it suffices to show the first equality.

According to Corollary~\ref{HQ51}, there exist some constants $c_1 > 0$ and $\alpha \in (0, 1]$ such that
\begin{eqnarray*}
c_1\, [\psi(x)]^{\alpha} &\ge& \mathrm{dist}(x, S) \quad \textrm{ for all } \quad x \in X.
\end{eqnarray*}
Let $L > 0$ be the Lipschitzian modulus of $f_0$ on $X,$ i.e., for all $x, x' \in X,$
\begin{eqnarray*}
|f_0(x) - f_0(x')| &\le& L \|x - x'\|.
\end{eqnarray*}
Also let $c_*$ be an arbitrary scalar greater than $c_1 L.$ Let $c > c_*$ and $x \in X$ an arbitrary point. Take any $\epsilon > 0.$ By definition, there exists $x' \in S$ such that 
\begin{eqnarray*}
\|x - x'\|  &\le&  \mathrm{dist}(x, S) + \epsilon.
\end{eqnarray*}
On the other hand, we have $f_0(x') \ge \inf_S f_0$ and $f_0(x) - f_0(x') \ge - L \|x - x'\|.$ Thus,
\begin{eqnarray*}
f_0(x) + c\, [\psi(x)]^{\alpha} 
&=& f_0(x') + \big( f_0(x) - f_0(x')\big) + c\, [\psi(x)]^{\alpha}  \\
&\ge& f_0(x') - L \|x - x'\| + \frac{c}{c_1} \,  \mathrm{dist}(x, S) \\
&\ge& f_0(x') - L \|x - x'\| + \frac{c}{c_1} \big (\|x - x'\|  - \epsilon \big)\\
&=& f_0(x') + \left(\frac{c}{c_1} - L \right) \|x - x'\|  - \frac{c}{c_1} \epsilon \\
&\ge& \inf_S  f_0  - \frac{c}{c_1} \epsilon.
\end{eqnarray*}
Taking the limit as $\epsilon \searrow 0$ yields the inequality
\begin{eqnarray*}
f_0(x) + c\, [\psi(x)]^{\alpha} &\ge& \inf_S  f_0.
\end{eqnarray*}
Since $x$ was arbitrary in $X,$ we get
\begin{eqnarray*}
\inf_{x \in X} \left\{f_0(x) + c\, [\psi(x)]^{\alpha} \right\} &\ge& \inf_{x \in S} f_0(x).
\end{eqnarray*}
The opposite inequality holds trivially, so the desired equality follows.
\end{proof}

\begin{remark}{\rm
(i) A semi-algebraic version of \cite[Theorem~2.1.2]{Luo1996} is derived directly from Theorem~\ref{DL54} because the conditions of the theorem hold automatically when $X$ is compact and $\psi$ is continuous. 

(ii) The conclusions of Theorem~\ref{DL54} can fail when ``semi-algebraic" is replaced by ``subanalytic''; see the next example.
}\end{remark}

\begin{example}\label{VD55} {\rm
Let $n := 2,$ $X := \left\{x := (x_1, x_2) \in \mathbb{R}^2 \mid |x_2| \leq 1\right\},$ $S := \left\{(x_1, x_2) \in \mathbb{R}^2 \mid x_2 = 0\right\}$ and consider the subanalytic functions $f_0(x) := -|x_2|$ and
\begin{eqnarray*}
\psi(x) &:=& 
\begin{cases}
e^{-\frac{1}{x_2^2}} & \textrm{ if } x_2 \ne 0, \\
0 & \textrm{ otherwise.}
\end{cases}
\end{eqnarray*}
Clearly, ${\rm dist} (x, S) = |x_2| \leq 1$ for all $x \in X$ and $f_*:= \inf_{x \in S} f_0(x) = 0.$ Moreover, it is easy to check that for any sequence of points $x_k \in X$ with $\psi(x_k) \to 0,$ one has $\mathrm{dist} (x_k, S) \to 0.$ However, the conclusions of Theorem~\ref{DL54} are not valid. This example also shows that \cite[Theorem~2.1.2]{Luo1996} can fail for optimization problems over unbounded sets.
}\end{example}

\begin{theorem} \label{DL55}
Assume that for any sequence of points $x_k$ in $X$ the following properties hold
\begin{enumerate}[{\rm ({A}1)}]
\item if $\psi(x_k) \to 0,$ then $\mathrm{dist}(x_k, S) \to 0;$
\item if $\mathrm{dist}(x_k, S) \to \infty,$ then $\psi(x_k) \to \infty.$
\end{enumerate}
Then there exist constants $c_* > 0, \alpha > 0,$ and $\beta > 0$ with $\alpha \le 1 \le \beta$ such that for all $c > c_*,$
\begin{eqnarray*}
\inf_{x \in S} f_0 (x) &=& \inf_{x \in X} \left\{ f_0(x) + c\big( [\psi(x)]^{\alpha}  + [\psi(x)]^{\beta} \big)\right\}, \\
\mathrm{argmin}_S f_0(x) &=& \mathrm{argmin}_{X} \left\{ f_0(x) + c\big( [\psi(x)]^{\alpha}  + [\psi(x)]^{\beta} \big)\right\}.
\end{eqnarray*}
\end{theorem}

\begin{proof}
The proof is similar to the one of Theorem~\ref{DL54} with taking into account Proposition~\ref{MD51} instead of Corollary~\ref{HQ51}.
\end{proof}

\begin{example}{\rm
Let $f_0 \colon \mathbb{R}^n\to \mathbb{R}$ be a Lipschitz function and
\begin{eqnarray*}
S &:=& \{x \in \mathbb{R}^n \mid g_1(x) \le 0, \ldots, g_m(x) \le 0 \},
\end{eqnarray*}
where $(g_1, \ldots, g_m) \colon \mathbb{R}^n \to \mathbb{R}^m$ is a {\em convenient and non-degenerate at infinity} polynomial mapping (see \cite{HaHV2017} for these concepts). Certainly the semi-algebraic function 
\begin{eqnarray*}
\psi \colon \mathbb{R}^n \to \mathbb{R}, \quad x \mapsto \max \{g_1(x), \ldots, g_m(x), 0\},
\end{eqnarray*}
is residual for $S.$ Moreover, in view of \cite[Theorem~3.7]{HaHV2017}, there exists a constant $c_* > 0$ such that
\begin{eqnarray*}
c_* \, \big([\psi(x)]^{\frac{1}{R(n + m - 1, d + 1)}} + [\psi(x)] \big) &\ge& \mathrm{dist}(x, S) \quad \textrm{ for all } \quad x \in \mathbb{R}^n,
\end{eqnarray*}
where $d$ is the maximum of the degrees of the polynomials $g_1, \ldots, g_m$ and (see \cite{Acunto2005})
\begin{eqnarray*}
R(a, b) &:=& b(3b - 3)^{a - 1} \quad \textrm{ for } a, b \ge 2.
\end{eqnarray*}
By a similar argument to the one given in the proof of Theorem~\ref{DL54}, we have for all $c > c_*$ that
\begin{eqnarray*}
\inf_{x \in S} f_0 (x) &=& \inf_{x \in \mathbb{R}^n} \left\{ f_0(x) + c\big( [\psi(x)]^{\frac{1}{R(n + m - 1, d + 1)}}  + [\psi(x)] \big)\right\}, \\
\mathrm{argmin}_S f_0(x) &=& \mathrm{argmin}_{\mathbb{R}^n} \left\{ f_0(x) + c\big( [\psi(x)]^{\frac{1}{R(n + m - 1, d + 1)}} + [\psi(x)] \big)\right\}
\end{eqnarray*}
(perhaps after increasing $c_*$).
}\end{example}

\section{Conclusions} \label{Section6}

In this paper, we have presented necessary and sufficient conditions for the exactness of penalty functions in optimization problems whose constraint sets are not necessarily bounded. The conditions are given in terms of properties of the objective and residual functions of the problems in question. 
Our results extend and improve some known facts in the literature on exact penalty functions.
It is our hope that this work will offer a platform to address problems with advanced features that are difficult to treat directly by the usual algorithms. 
For details of applying penalty functions to realistic applications, we refer the reader to the monographs \cite{CuiPang2021, Fiacco1990, Hiriart-Urruty1993book, Luo1996}.

\subsection*{Acknowledgments} 
A part of this work was done while the second author and the third author were visiting Academy for Advanced Interdisciplinary Studies, Northeast Normal University, Changchun, China. The last version of the paper was completed when the first author and the second author were visiting the Vietnam Institute for Advanced Study in Mathematics (VIASM). The authors would like to thank these organizations for their hospitality and support.

\subsection*{Data availability} There is no data included in this paper.

\subsection*{Declarations} No potential conflict of interest was reported by the authors. 


\end{document}